\documentclass[preprint,hidelinks,11pt]{elsarticle}

\usepackage{psfrag,graphicx,amsmath,amsthm,amssymb,hyperref,color}
\journal{}

\newtheorem{example}{Example}[section]
\newtheorem{remark}{Remark}[section]
\newtheorem{theorem}{Theorem}[section]

\newtheorem{lemma}{Lemma}[section]
\numberwithin{equation}{section}

\DeclareFontEncoding{FMS}{}{}
\DeclareFontSubstitution{FMS}{futm}{m}{n}
\DeclareFontEncoding{FMX}{}{}
\DeclareFontSubstitution{FMX}{futm}{m}{n}
\DeclareSymbolFont{fouriersymbols}{FMS}{futm}{m}{n}
\DeclareSymbolFont{fourierlargesymbols}{FMX}{futm}{m}{n}
\DeclareMathDelimiter{\tbar}{\mathord}{fouriersymbols}{152}{fourierlargesymbols}{147}

\newcommand{\jump}[1]{\left[\!\left[#1\right]\!\right]}
\newcommand{\bigjump}[1]{\left[\!\!\left[#1\right]\!\!\right]}

\begin{document}

\begin{frontmatter}



\title{Superconvergence of Immersed Finite Element Methods for Interface Problems}

\author{Waixiang Cao\fnref{label1}}
\ead{wxcao@csrc.ac.cn}

\author{Xu Zhang\fnref{label2}}
\ead{xuzhang@math.msstate.edu}

\author{Zhimin Zhang\fnref{label1,label3}}
\ead{zzhang@math.wayne.edu}

\address[label1]{Beijing Computational Science Research Center, Beijing, 100094, China}
\address[label2]{Department of Mathematics and Statistics, Mississippi State University, Mississippi State MS 39762}
\address[label3]{Department of Mathematics, Wayne State University, Detroit, MI 48202}

\fntext[label3]{This work is supported in part by
the China Postdoctoral Science Foundation 2015M570026, and the National Natural Science Foundation of China
(NSFC) under grants No. 91430216, 11471031, 11501026,  and the US
National Science Foundation (NSF) through grant DMS-1419040.}
\begin{abstract}
In this article, we study superconvergence properties of immersed finite element methods for the one dimensional elliptic interface problem. Due to low global regularity of the solution, classical superconvergence phenomenon for finite element methods disappears unless the discontinuity of the coefficient is resolved by partition. We show that immersed finite element solutions inherit all desired superconvergence properties from standard finite element methods without requiring the mesh to be aligned with the interface. In particular, on interface elements, superconvergence occurs at roots of generalized orthogonal polynomials that satisfy both orthogonality and interface jump conditions.
\end{abstract}


\end{frontmatter}

\section{Introduction}
\label{sec: intro}

Immersed finite element (IFE) methods are a class of finite element methods (FEM) for solving differential equations with discontinuous coefficients, often known as interface problems. Unlike the classical FEM whose mesh is required to be aligned with the interface, IFE methods do not have such restriction. Consequently, IFE methods can use more structured, or even uniform meshes to solve interface problem regardless of interface location. This flexibility is advantageous for problems with complicated interfacial geometry \cite{2010VallaghePapadopoulo} or for dynamic simulation involving a moving interface \cite{2013HeLinLinZhang, 2013LinLinZhang2, 2013LinLinZhang1}.

The main idea of IFE methods is to adapt approximating functions instead of meshes to fit the interface. On elements containing (part of) the interface, which we call interface elements, universal polynomials cannot approximate the solution accurately  because of the low regularity of solution at the interface. A simple remedy is to construct piecewise polynomials as basis functions on interface elements in order to mimic the exact solution. The first IFE method was developed by Li \cite{1998Li} for solving the one-dimensional two-point boundary value problem. Piecewise linear shape functions were constructed on interface elements to incorporate the interface jump conditions. Following this idea, a family of quadratic IFE functions were introduced in \cite{2006CampLinLinSun}. Later in \cite{2007AdjeridLin, 2009AdjeridLin}, Adjerid and Lin extended the IFE approximation to arbitrary polynomial degree, and proved the optimal error estimates in the energy and the $L^2$-norms. In the past decade, IFE methods have also been extensively studied for a variety of interface problems in two dimension \cite{2007GongLiLi, 2004LiLinLinRogers, 2003LiLinWu, 2013LinSheenZhang,2015LinYangZhang1,2015LinYangZhang2} and three dimension \cite{2005KafafyLinLinWang,2010VallaghePapadopoulo}.

There have been many studies in the mathematical theories for IFE methods, for example \cite{2009AdjeridLin, 2010ChouKwakWee, 2008HeLinLin, 1998Li, 2015LinLinZhang}. Most of theoretical analysis focuses on error estimation in Sobolev $H^1$- and $L^2$- norms, but very few literature are concerned with the pointwise convergence. To the best of our knowledge, there is no systematic study on superconvergence phenomenon of IFE methods. Superconvergence theory for classical finite element methods \cite{2001BabuskaStrouboulis,1974DouglasDupont,1995Wahlbin} are invalid for IFE methods, unless the discontinuity of coefficient is resolved by the solution mesh.

Superconvergence phenomena of FEM were discussed as early as 1967 by Zienkiewicz and Cheung \cite{Zienkiewicz-Cheung}.  Later, Douglas and Dupont in \cite{1974DouglasDupont} proved that the $p$-th order $C^0$ finite element method to the two-point boundary value problem converges with rate $O(h^{2p})$ at nodal points.
Since then the superconvergence behavior of FEM has been studied intensively.  We refer to   \cite{Babuska1996,Bramble.Schatz.math.com,Chen.C.M2012,Neittaanmaki1987,V.Thomee.math.comp,1995Wahlbin}  for an incomplete list of references. In the mean time, there also has been considerable interest in studying superconvergence for other numerical methods, for example, spectral and spectral collocation methods \cite{zhang1, zhang2, zhang3},
finite volume methods \cite{Cai.Z1991, Cao;Zhang;Zou2012,Cao;zhang;zou:2kFVM, Chou_Ye2007,Xu.J.Zou.Q2009}, discontinuous Galerkin and local discontinuous Galerkin methods \cite{Adjerid;Massey2006,Cao-Shu-zhang-Yang2D,
   Cao;zhang:supLDG2k+1,Cao;zhang;zou:2k+1,Guo_zhong_Qiu2013,Xie;Zhang2012,Yang;Shu:SIAM2012}.

In this article, we focus on the conforming $p$-th degree IFE methods for the prototypical one-dimensional elliptic interface problem. There are two major contributions in this article. First, we present a novel approach for developing IFE basis functions. The idea is completely different from classical approaches \cite{2007AdjeridLin,2009AdjeridLin}, and the construction is based on the theory of orthogonal polynomials. Our new IFE bases accommodate interface jump conditions, and they satisfy certain orthogonality conditions which will be specified later. These basis functions can be explicitly constructed without solving linear systems. In an interface element, these IFE bases are either polynomials or piecewise polynomials, hence we call them \textit{generalized orthogonal polynomials}.

Next, we analyze superconvergence properties of IFE methods. We will show that superconvergence phenomena occur at the roots of generalized orthogonal polynomials. To be more specific, the convergence rate of $p$-th degree IFE solutions is $O(h^{p+2})$ at nodal points. The accuracy at nodes can be improved to \textit{exact} if the elliptic operator has only the diffusion term. The IFE solution converge to the exact solution with rate $O(h^{p+2})$ at the roots of \textit{generalized Lobatto polynomials}, and the convergence rate of derivatives is escalated to $O(h^{p+1})$ at the roots of \textit{generalized Legendre polynomials}. All the results can be viewed as an extension from the classic result for FEM \cite{1974DouglasDupont}.

The rest of the paper is organized as follows. In Section 2, we recall the IFE methods for interface problems and introduce some notations. In Section 3, we introduce the generalized orthogonal polynomials, based on which we present an explicit approach to construct IFE basis functions. In Section 4, we study the superconvergence properties of IFE methods for interface problems. In Section 5, we report some numerical results. A few concluding remarks are presented in Section 6.

\section{Immersed Finite Element Methods}
Let $\Omega = (a,b)$ be an open interval. Assume that $\alpha\in \Omega$ is an interface point such that $\Omega^- = (a,\alpha)$ and $\Omega^+ = (\alpha,b)$. Consider the following one-dimensional elliptic interface problem
\begin{equation}\label{eq: DE}
  -(\beta u')' + \gamma u' + cu = f, ~~~x\in \Omega^-\cup\Omega^+,
\end{equation}
\begin{equation}\label{eq: BC}
  u(a) = u(b) = 0.
\end{equation}
The diffusion coefficient $\beta$ is assumed to have a finite jump across the interface $\alpha$. Without loss of generality, we assume that $\beta$ is a piecewise constant defined by
\begin{equation}\label{eq: jump}
  \beta(x) =
  \left\{
    \begin{array}{ll}
      \beta^-, & \text{if}~ x\in\Omega^-, \\
      \beta^+, & \text{if}~ x\in\Omega^+,
    \end{array}
  \right.
\end{equation}
where $\min\{\beta^+,\beta^-\}>0$. The coefficients $\gamma$ and $c$ are assumed to be constants. At the interface $\alpha$, the solution is assumed to satisfy the interface jump conditions
\begin{equation}\label{eq: jump condition}
  \jump{u(\alpha)} = 0,~~~
  \bigjump{\beta u'(\alpha)} = 0,
\end{equation}
where $\jump{v(\alpha)} := v(\alpha^+) - v(\alpha^-)$. Denote the ratio of coefficient jump by $\rho = \frac{\beta_{\max}}{\beta_{\min}}$ where $\beta_{\max} = \max\{\beta^+,\beta^-\}$,  $\beta_{\min} = \min\{\beta^+,\beta^-\}$\\

Throughout this article, we use standard notation of Sobolev spaces. We will also need to develop a few new spaces that characterize the interface problems. We define for $m\ge 1$ and $q\ge 1$ the Sobolev space
\begin{eqnarray}
  \tilde W_\beta^{m,q}(\Omega)= \Big\{v\in C(\Omega)\!\!\!\!&:&\!\!\! v|_{\Omega^\pm}\in W^{m,q}(\Omega^\pm),~v|_{\partial\Omega}=0, \nonumber\\
   && \!\!\!\bigjump{\beta v^{(j)}(\alpha)} = 0,~ j = 1,2,\cdots, m\Big\}, \label{eq: space2}
\end{eqnarray}
equipped the norm and semi-norm
\begin{equation*}
\|v\|_{m,q,\Omega}^q=\|v\|_{m,q,\Omega^-}^q+\|v\|_{m,q,\Omega^+}^q,~~~
|v|_{m,q,\Omega}^q=|v|_{m,q,\Omega^-}^q+|v|_{m,q,\Omega^+}^q,
\end{equation*}
for $q<\infty$, and 
\begin{equation*}
\|v\|_{m,\infty,\Omega}=\max\{\|v\|_{m,\infty,\Omega^-},\|v\|_{m,\infty,\Omega^+}\},~~~
|v|_{m,\infty,\Omega}=\max\{|v|_{m,\infty,\Omega^-},|v|_{m,\infty,\Omega^+}\}.
\end{equation*}

On a subset $\Lambda\subset \Omega$ that contains the interface point $\alpha$, we define
\[
     \|v\|_{m,q,\Lambda}^q=\|v\|_{m,q,\Lambda^-}^q+\|v\|_{m,q,\Lambda^+}^q,\ \  
     |v|_{m,q,\Lambda}^q=|v|_{m,q,\Lambda^-}^q+|v|_{m,q,\Lambda^+}^q,
 \] 
where $\Lambda^\pm = \Lambda\cap\Omega^\pm$. If $\Lambda=\Omega$, we usually write $\|\cdot\|_{m,q}$ instead of $\|\cdot\|_{m,q,\Omega}$ for simplicity. In addition, if $q=2$, we simply write $\|\cdot\|_{m}$ instead of $\|\cdot\|_{m,q}$. \\
%We use orthogonal polynomials to construct basis functions on each noninterface element $\tau_i$. To be more specific, we use standard Lobatto polynomials $\{\psi_n\}$ defined in \eqref{eq: Labotto Poly}. Accordingly, the $p$-th local FE space on $\tau_i$ is the standard polynomial space of degree $p$, denoted by ${\mathbb{P}}_p(\tau_i)$. On the interface element $\tau_k$, we construct IFE basis function using the generalized Lobatto polynomials defined in \eqref{eq: lobatto 0} - \eqref{eq: lobatto n}. The $p$-th degree local IFE space on $\tau_k$, denoted by $\tilde{\mathbb{P}}_p(\tau_k)$ {\color{red} will be constructed according. see }
%We define the global IFE space $S_p(\mathcal{T}_h)$ as
%\begin{equation*}
%S_p(\mathcal{T}_h) = \{v\in H^1_0(\Omega): v|_{\tau_i}\in \mathbb{P}_p(\tau_i), \forall i\neq k;~ v|_{\tau_k}\in \tilde{\mathbb{P}}_p(\tau_k)\}
%\end{equation*}
%The IFEM for \eqref{eq: DE}-\eqref{eq: jump condition} is to find $u_h\in S_p(\mathcal{T}_h) $ such that
%\begin{equation*}
%(\beta u_h',v_h') + (\gamma u_h', v_h)  + (cu_h,v_h) = (f,v_h),~~~\forall v_h\in S_p(\mathcal{T}_h) ,
%\end{equation*}
%where $(\cdot,\cdot)$ is the standard $L^2$ inner product on $(a,b)$.
 
Next, we recall the main idea of the immersed finite element methods (IFEM) for interface problem \eqref{eq: DE} - \eqref{eq: jump condition}. Consider the following interface-independent partition of $\Omega$:
\begin{equation}\label{eq: partition}
  a = x_0 < x_1< \cdots <x_{k-1}\leq\alpha\leq x_{k}<\cdots<x_N = b.
\end{equation}
Based on the partition \eqref{eq: partition}, we define a mesh $\mathcal{T}_h = \{\tau_i\}_{i = 1}^N$, where $\tau_i = (x_{i-1},x_i)$. Denoted by $h_i = x_i-x_{i-1}$ the size of the element $\tau_i$, and by $h=\max\{h_i,i=1,\cdots, N\}$ the mesh size of $\mathcal{T}_h$. Note that the interface $\alpha$ is located in the element $\tau_k$, which we call the interface element. The rest of elements $\tau_i$, $i\ne k$ are called noninterface elements. If the interface $\alpha$ coincides with the mesh point $x_{k-1}$ or $x_{k}$, then the partition \eqref{eq: partition} becomes interface-fitted; hence there is no difference between the IFEM and standard FEM. 

Standard polynomials are used to as basis functions on all noninterface elements. To be more specific, we use the standard Lobatto polynomials as bases. The $p$-th degree FE space on the noninterface element $\tau_i$ is the standard polynomial space of degree $p$,  denoted by ${\mathbb{P}}_p(\tau_i)$. On the interface element $\tau_k$, we construct new IFE basis functions using the generalized Lobatto polynomials (will be defined in \eqref{eq: lobatto 0} - \eqref{eq: lobatto n}). The corresponding $p$-th degree IFE space on $\tau_k$ is denoted by $\tilde{\mathbb{P}}_p(\tau_k)$ shall be defined in \eqref{eq: interface Space} .

We define the $p$-th degree global IFE space on the mesh $\mathcal{T}_h$ by
\begin{equation*}
S_p(\mathcal{T}_h) = \{v\in H^1_0(\Omega): v|_{\tau_i}\in \mathbb{P}_p(\tau_i), \forall i\neq k;~ v|_{\tau_k}\in \tilde{\mathbb{P}}_p(\tau_k)\}
\end{equation*}
The IFEM for \eqref{eq: DE}-\eqref{eq: jump condition} is to find $u_h\in S_p(\mathcal{T}_h) $ such that
\begin{equation*}
(\beta u_h',v_h') + (\gamma u_h', v_h)  + (cu_h,v_h) = (f,v_h),~~~\forall v_h\in S_p(\mathcal{T}_h) ,
\end{equation*}
where $(\cdot,\cdot)$ is the standard $L^2$ inner product on $(a,b)$.

\section{Generalized Orthogonal Polynomials}
In this section, we recall standard Legendre and Lobatto polynomials, and use them as basis functions on noninterface elements. Next, we construct the generalized orthogonal polynomials to be used as basis functions on interface elements. 

\subsection{Standard Orthogonal Polynomials}
As usual, we construct basis functions on the reference interval $\tau = [-1,1]$, then map them to each physical element $\tau_i$ by appropriate affine mapping. Let $P_n(\xi)$ be the Legendre polynomial of degree $n$ on $\tau$ defined by
\begin{equation*}
P_n(\xi) = \frac{1}{2^n n!}\frac{d^n}{d\xi^n}\big[(\xi^2-1)^n\big]
\end{equation*}
Legendre polynomials satisfy the following orthogonality
\begin{equation}\label{eq: Legendre orthogonality}
  %(P_m,P_n) :=
  \int_{-1}^1 P_m(\xi)P_n(\xi) d\xi = \frac{2}{2n+1}\delta_{mn}.
\end{equation}
Define $\{\psi_n\}$ {to be} the family of Lobatto polynomials on $\tau =[-1,1]$,
\begin{equation}\label{eq: Labotto Poly}
  \psi_0(\xi) = \frac{1-\xi}{2},~~~
  \psi_1(\xi) = \frac{1+\xi}{2},~~~
  \psi_n(\xi) = \int_{-1}^\xi P_{n-1}(t)dt, ~~~ n\geq 2.
\end{equation}

\subsection{Generalized Orthogonal Polynomials}
On the interface element $\tau_k$ containing $\alpha$, we construct a sequence of polynomials satisfying both orthogonality and interface jump conditions. Again, we map $\tau_k$ to the reference interval $\tau = [-1,1]$ containing the reference interface point $\hat\alpha$. Let $\hat\beta(\xi) = \beta(x)$ such that $\hat\beta(\xi) = \beta^-$ on $\tau^-=(-1,\hat\alpha)$ and $\hat\beta(\xi) = \beta^+$ on $\tau^+=(\hat\alpha,1)$.

Define a sequence of orthogonal polynomials $\{L_n\}$ with the weight function $w(\xi) = \frac{1}{\hat\beta(\xi)}$, \emph{i.e.},
\begin{equation}\label{eq: general Legredre orthogonality}
  (L_n,L_m)_w :=\int_{-1}^1 w(\xi) L_n(\xi) L_m(\xi) d\xi = c_n\delta_{mn},
\end{equation}
where $c_n = \|L_n\|_w^2 = (L_n,L_n)_w$. If we require $\{L_n\}$ to be monic polynomials, then they can be uniquely constructed via the following three-term recurrence formula (\cite{2011ShenTangWang}, Theorem 3.1):
\begin{remark}
Let $\{{L}_n\}$ be the family of monic orthogonal polynomials satisfying \eqref{eq: general Legredre orthogonality}. Then $\{{L}_n\}$ can be constructed as follows
\begin{equation}\label{eq: monic generalized Legendre poly 0 1}
  {L}_0(\xi) = 1,~~~{L}_1(\xi) = \xi - a_0, ~~~
\end{equation}
\begin{equation}\label{eq: monic generalized Legendre poly}
  {L}_{n+1}(\xi) = (\xi-a_n){L}_n(\xi) - b_n{L}_{n-1}(\xi),~ ~n\ge 1,
\end{equation}
where
\begin{eqnarray*}
  a_n &=& \frac{(\xi{L}_n,{L}_n)_w}{({L}_n,{L}_n)_w},~~~n\geq 0 \\
  b_n &=& \frac{({L}_n,{L}_n)_w}{({L}_{n-1},{L}_{n-1})_w},~~~n\geq 1.
\end{eqnarray*}
\end{remark}
The polynomials $\{L_n\}$ are generalized from standard Legendre polynomials $\{P_n\}$ by allowing the weight function to be discontinuous. Hence, we call $\{L_n\}$ the \textit{generalized Legendre polynomials}.

%\subsection{Generalized Lobatto Polynomials}
Next, we define a sequence of piecewise polynomial $\{\phi_n\}$ in a similar manner as \eqref{eq: Labotto Poly}
\begin{eqnarray}
  \phi_0(\xi) &=& \left\{\begin{array}{ll}
      \frac{(1-\hat\alpha)\beta^-+(\hat\alpha-\xi)\beta^+}{(1-\hat\alpha)\beta^-+(1+\hat\alpha)\beta^+}, & \text{in}~\tau^-, \vspace{1mm}\\
      \frac{(1-\xi)\beta^-}{(1-\hat\alpha)\beta^-+(1+\hat\alpha)\beta^+}, & \text{in}~ \tau^+.
    \end{array}\right. \label{eq: lobatto 0}\\
  \phi_1(\xi) &=& \left\{
  \begin{array}{ll}
      \frac{(1+\xi)\beta^+}{(1-\hat\alpha)\beta^-+(1+\hat\alpha)\beta^+}, & \text{in}~\tau^-, \vspace{1mm}\\
      \frac{(\xi-\hat\alpha)\beta^-+(1+\hat\alpha)\beta^+}{(1-\hat\alpha)\beta^-+(1+\hat\alpha)\beta^+}, & \text{in}~\tau^+.
    \end{array}\right. \label{eq: lobatto 1}\\
  \phi_{n}(\xi) &=&  \int_{-1}^\xi w(t) L_{n-1}(t)dt,~~~n = 2,3,\cdots\label{eq: lobatto n}
\end{eqnarray}
Note that $\phi_0$ and $\phi_1$ are constructed to fulfill nodal value conditions
\begin{equation*}
  \phi_0(-1) = 1, ~~\phi_0(1) = 0,~~
  \phi_1(-1) = 0, ~~\phi_1(1) = 1.
\end{equation*}
and the interface jump condition \eqref{eq: jump condition}. In fact, $\phi_0$ and $\phi_1$ are piecewise linear polynomials, and they are exactly the two Lagrange type IFE nodal basis functions (see \cite{2009AdjeridLin, 1998Li}).

\begin{theorem}
$\{\phi_n\}$ is a sequence of piecewise polynomials and satisfy
\begin{itemize}
  \item the interface jump conditions
\begin{equation}\label{eq: interface condition}
  \jump{\phi_n(\hat\alpha)} = 0,~~~~\bigjump{\hat\beta\phi_n'(\hat\alpha)} = 0, ~~~\forall n\geq 0,
\end{equation}
  \item the weighted orthogonality condition
\begin{equation}\label{eq: general orthogonal condition}
  \langle \phi_m,\phi_n\rangle_{\hat\beta} := \int_{-1}^1 \hat\beta(\xi) \phi_m'(\xi)\phi_n'(\xi)d\xi = \tilde c_n\delta_{mn},~~~\forall m,n\ge 1,
\end{equation}
where $\tilde c_n$ is some nonzero constant.
\end{itemize}
\end{theorem}
\begin{proof}
We first prove the interface jump conditions \eqref{eq: interface condition}. It is true for $\phi_0$ and $\phi_1$ by direct verification using \eqref{eq: lobatto 0} and \eqref{eq: lobatto 1}. For $n\ge 2$, we note that $\phi_n$ is continuous because it is defined through the integral \eqref{eq: lobatto n}. Moreover, since $\{L_n\}$ is a sequence of polynomials, then
\begin{equation*}
  \bigjump{\hat\beta\phi_n'(\hat\alpha)} = \beta^+\phi_n'(\hat\alpha+) - \beta^-\phi_n'(\hat\alpha-) =  L_{n-1}(\hat\alpha+) - L_{n-1}(\hat\alpha-) = 0.
\end{equation*}

The orthogonality \eqref{eq: general orthogonal condition} follows from \eqref{eq: general Legredre orthogonality} and \eqref{eq: lobatto n}, \emph{i.e.},
\begin{equation*}
  \langle \phi_m,\phi_n\rangle_{\hat\beta} = (L_{m-1}, L_{n-1})_w = c_{n-1} \delta_{m-1,n-1} = \tilde c_n\delta_{mn}.
\end{equation*}
\end{proof}

The piecewise polynomials $\{\phi_n\}$ are generalized from standard Lobatto polynomials $\{\psi_n\}$ defined in \eqref{eq: Labotto Poly}. The construction \eqref{eq: lobatto n} uses piecewise constant weight function $w(\xi) = \frac{1}{\hat\beta(\xi)}$ instead of a universal constant one. We call $\{\phi_n\}$ the \textit{generalized Lobatto polynomials}.

The generalized Lobatto polynomials $\{\phi_n\}$ form a sequence of IFE basis functions satisfying both interface jump conditions and orthogonal conditions. In Figure \ref{fig: lobatto IFE basis}, we plot a few generalized Legendre polynomials $L_n$ and generalized Lobatto polynomials $\phi_n$ for the configuration of $\hat\alpha = 0.15$ and $\hat\beta = \{1,5\}$. In Figure \ref{fig: lobatto IFE basis 2pt}, we plot the generalized polynomials for multiple (two) interface points $\hat\alpha = -0.15$ and $0.4$. The coefficient $\hat\beta$ has three pieces in this case, \emph{i.e.},  $\hat\beta = \{1,5,3\}$.
\begin{figure}[thb]
  \centering
  \includegraphics[width=.49\textwidth]{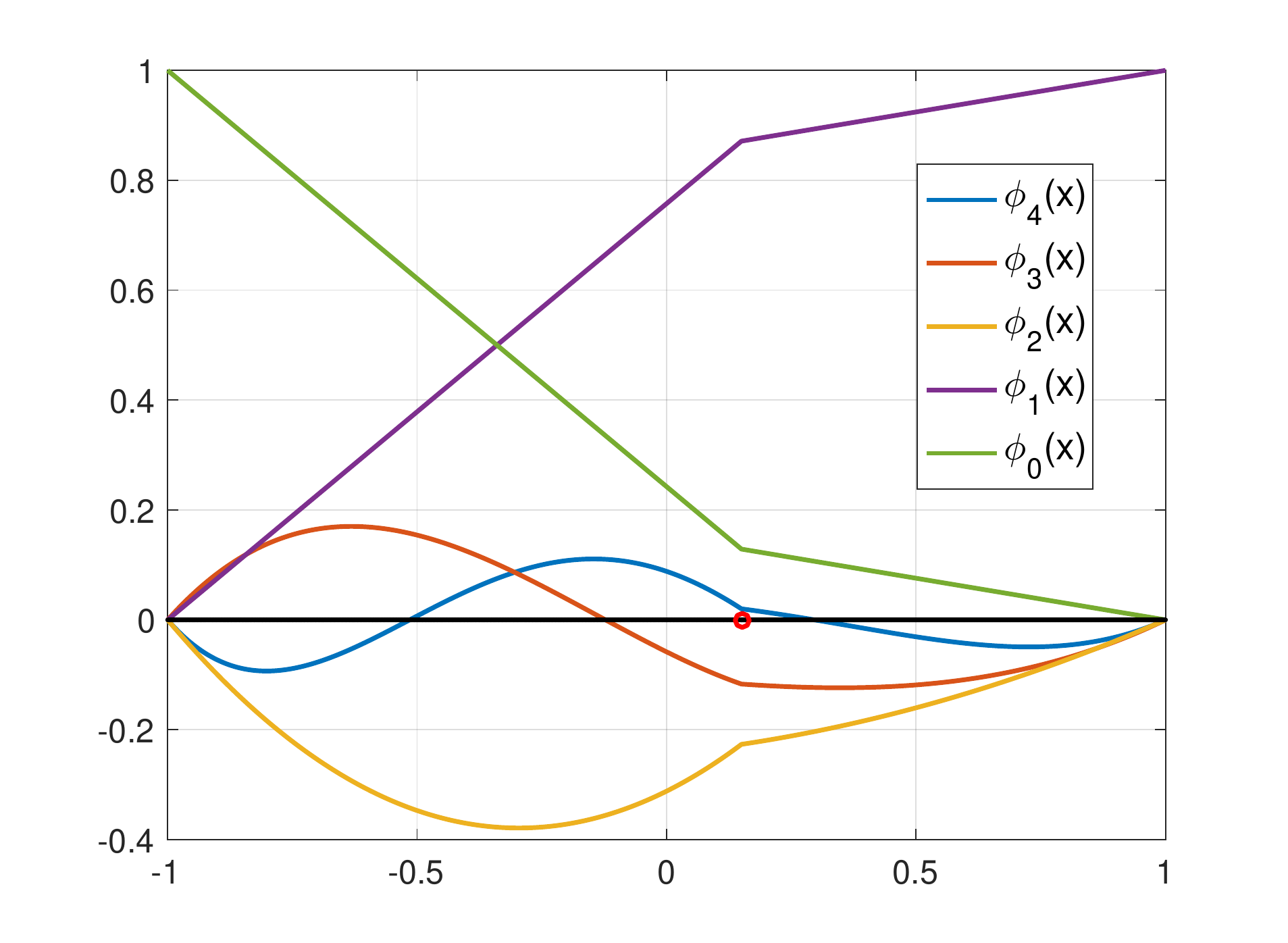}
  \includegraphics[width=.49\textwidth]{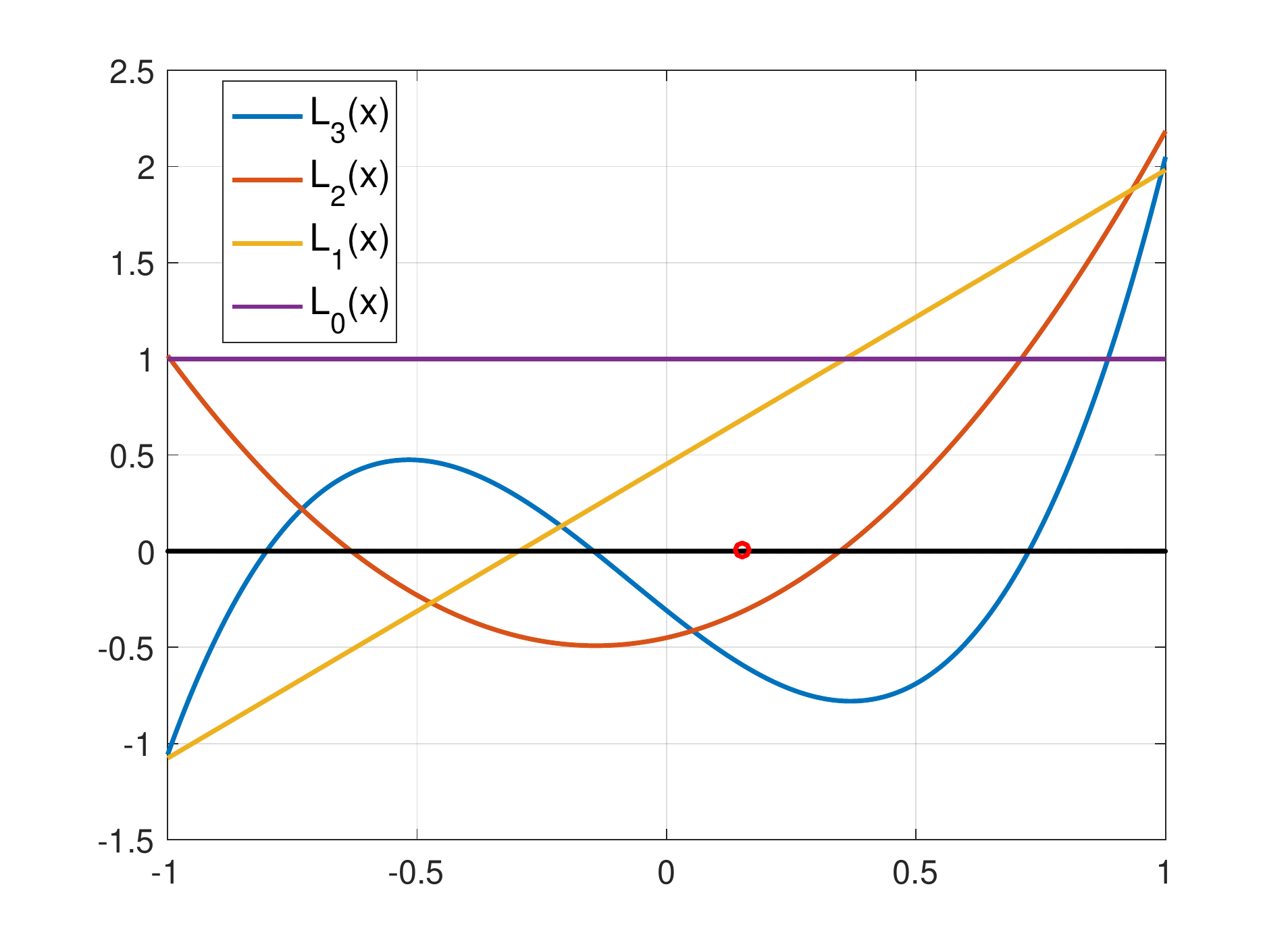}\\
  \caption{Generalized Lobatto (left) and Legendre (right) polynomials with one interface point}
  \label{fig: lobatto IFE basis}
\end{figure}

\begin{figure}[thb]
  \centering
  \includegraphics[width=.49\textwidth]{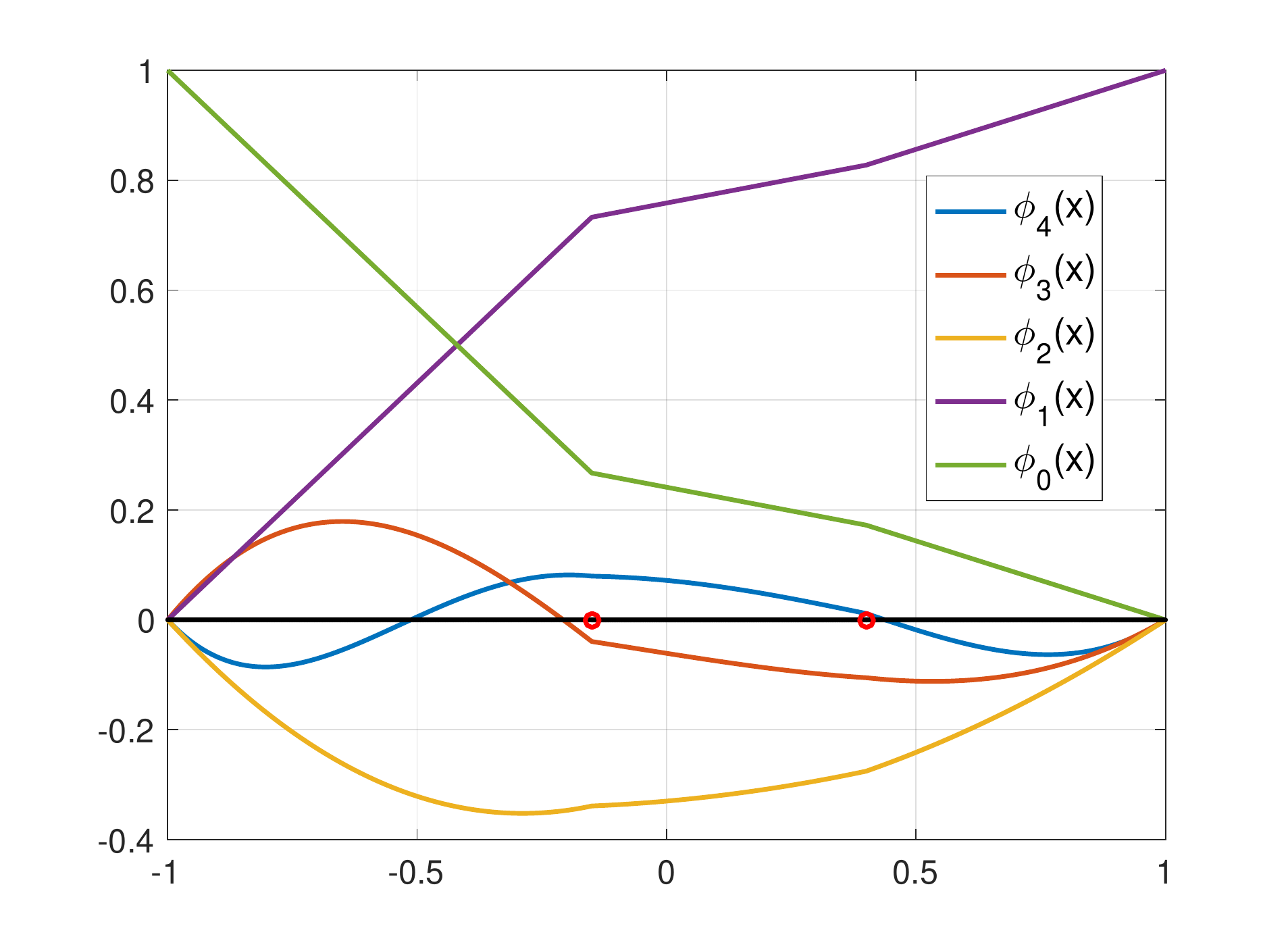}
  \includegraphics[width=.49\textwidth]{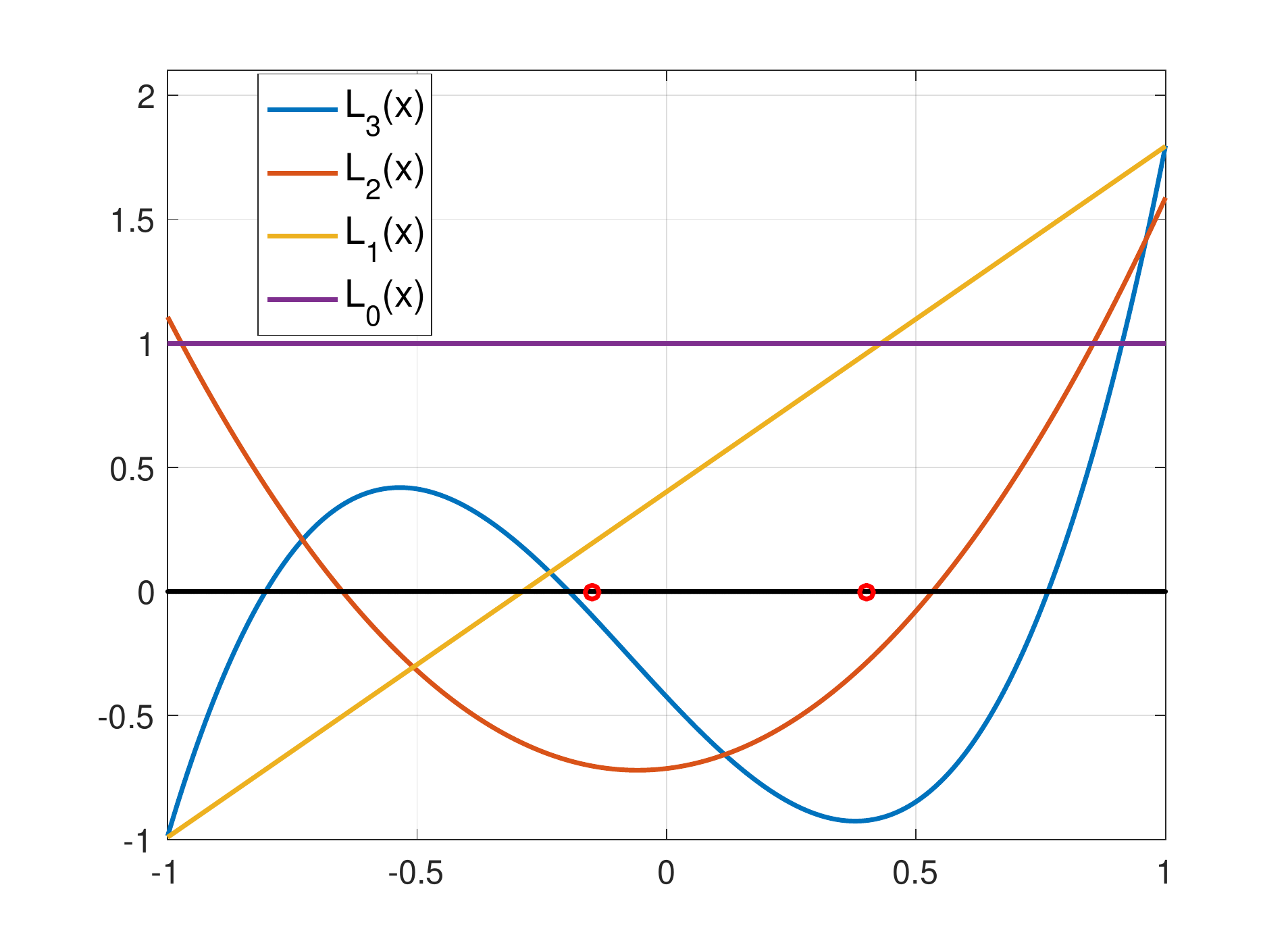}\\
  \caption{Generalized Lobatto (left) and Legendre (right) polynomials with two interface points}
  \label{fig: lobatto IFE basis 2pt}
\end{figure}

\begin{remark}
The generalized Lobatto polynomials $\{\phi_n\}$ are identical (up to a multiple constant) to IFE basis functions introduced in \cite{2007AdjeridLin}. However, the construction in this article is more explicit and does not require solving a linear system. This procedure is more advantageous when there are multiple discontinuities in an interval.
\end{remark}

\begin{remark}
In the construction procedure of $\phi_n$, we did not impose the \textit{extended} interface jump conditions \cite{2007AdjeridLin}:
\begin{equation}\label{eq: extra jump condition}
  \bigjump{\hat\beta \phi_n^{(j)}(\hat\alpha)} = 0,~~~~\forall~ j = 2,3, \cdots, n.
\end{equation}
However, it can be easily verified that all the generalized Lobatto polynomials $\{\phi_n\}$ satisfy \eqref{eq: extra jump condition} automatically.
\end{remark}

We can obtain the local FE basis functions $\psi_{i,n}$ on each noninterface element $\tau_i$ and the IFE basic functions $\phi_{k,n}$ on the interface element $\tau_k$ by the following affine mappings,
\begin{equation}\label{eq: Lobatto Noninterface}
  \psi_{i,n}(x) := \psi_n(\xi) = \psi_n\left(\frac{2x-x_{i-1}-x_i}{h_{i}}\right),~~~~n \ge 0.
\end{equation}
\begin{equation}\label{eq: Lobatto interface}
  \phi_{k,n}(x) := \phi_n(\xi) = \phi_n\left(\frac{2x-x_{k-1}-x_k}{h_{k}}\right),~~~~n \ge 0.
\end{equation}

Then the $p$-th degree local FE space $\mathbb{P}_p(\tau_i) $ on noninterface elements $\tau_i$, $i\neq k$, and IFE space $\tilde{\mathbb{P}}_p(\tau_k)$ on interface element $\tau_k$ are defined by
\begin{equation}\label{eq: Noninterface Space}
  \mathbb{P}_p(\tau_i) = \text{span}\{\psi_{i,n}: n = 0,1,\cdots, p\}.
\end{equation}
\begin{equation}\label{eq: interface Space}
  \tilde{\mathbb{P}}_p(\tau_k) = \text{span}\{\phi_{k,n}: n = 0,1,\cdots, p\}.
\end{equation}
Finally, the $p$-th degree global IFE space is defined by
\begin{equation}\label{eq: IFE space}
  S_p(\mathcal{T}_h) := \{v\in H_0^1(\Omega): v|_{\tau_i} \in \mathbb{P}_p(\tau_i), ~i\ne k, ~\text{and}~ v|_{\tau_k} \in \tilde{\mathbb{P}}_p(\tau_k)\}.
\end{equation}
The IFEM for the interface problem \eqref{eq: DE} - \eqref{eq: jump condition} is: find $u_h\in S_p(\mathcal{T}_h)$ such that
\begin{equation}\label{eq: IFE method}
  a(u_h,v_h):= (\beta u_h', v_h')+ (\gamma u_h',v_h) + (cu_h,v_h) = (f,v_h),~~~~\forall v_h\in S_p(\mathcal{T}_h).
\end{equation}

\subsection{Properties of Generalized Orthogonal Polynomials}
In this subsection, we investigate some fundamental properties of the generalized orthogonal polynomials.

First, it is interesting to know the number and distribution of zeros for the generalized Lobatto polynomials and generalized Legendre polynomials in the interval $[-1,1]$. To prove our main result, we need the following lemma.
\begin{lemma} (Generalized Rolle's theorem)
\label{lemma: gen Rolle}
Assume that the function $f$ is real-valued and continuous on a closed interval $[a,b]$ with $f(a)=f(b)$. If for every $x$ in the open interval $(a,b)$, both of one side limits
\begin{equation*}
  f'(x+) = \lim_{h\to 0^+}\frac{f(x+h)-f(x)}{h},~~~~
  f'(x-) = \lim_{h\to 0^-}\frac{f(x+h)-f(x)}{h}
\end{equation*}
exist, % in the extended real line $[-\infty,\infty]$,
then there is some number $c$ in the open interval $(a,b)$ such that one of the two limits $f'(c+)$ and $f'(c-)$ is $\ge 0$ and the other is $\leq 0$.
%If the right and left hand limits agree for every $x$, then they agree in particular $c$, hence the derivative of $f$ exists at $c$ and is equal to zero.
\end{lemma}
The above lemma generalizes the Rolle's theorem to functions that are continuous on $[a,b]$, but not necessarily differentiable at all interior points of $(a,b)$. The proof is straightforward and similar to the standard Rolle's theorem; hence we omit it in this article. 

\begin{theorem}\label{theo2:2}
The generalized Legendre polynomials $\{L_n\}$ and generalized Lobatto polynomials $\{\phi_n\}$ have the same numbers of roots as the standard Legendre polynomials $\{P_n\}$ and Lobatto polynomials $\{\psi_n\}$, respectively, \emph{i.e.},
\begin{enumerate}
  \item For $n\geq 1$, $L_n$ has $n$ simple roots in the open interval $(-1,1)$.
  \item For $n\ge 1$, $\phi_{n+1}(\pm 1) = 0$, and $\phi_{n+1}$ has $n-1$ simple ``roots" in the open interval $(-1,1)$, \emph{i.e.}, the piecewise polynomial $\phi_{n+1}(\xi)$ crosses the $\xi$-axis $n-1$ times in $(-1,1)$.
\end{enumerate}
\end{theorem}

\begin{proof}
Note that $\{L_n\}$ is a family of orthogonal polynomials on $[-1,1]$. The weight function $w(\xi) = \hat\beta(\xi)^{-1}$ is positive and is a Lebesgue integrable function. Hence, the polynomial $L_n$ has $n$ simple roots in $(-1,1)$.\\

For the generalized Lobatto polynomial $\phi_{n+1}$, by its definition \eqref{eq: lobatto 1}, it is obvious that $\phi_{n+1}(-1) = 0$. The orthogonality condition \eqref{eq: general Legredre orthogonality} yields
\begin{equation*}
  \phi_{n+1}(1) = \int_{-1}^{1}w(\xi){L}_{n}(\xi)d\xi = \int_{-1}^{1}w(\xi){L}_{n}(\xi) {L}_{0}(\xi) d\xi = 0.
\end{equation*}

In the remaining of the proof, we will show that $\phi_{n+1}$ has exactly $n-1$ roots in the open interval $(-1,1)$.
By \eqref{eq: interface condition} and \eqref{eq: general orthogonal condition}, we have for $m\leq n$,
\begin{eqnarray*}
  \int_{-1}^1 \hat\beta\phi_{n+1}'(\xi) \phi_{m}'(\xi) d\xi &=&
   - \int_{-1}^1 \phi_{n+1}(\xi) (\hat\beta\phi_{m}')'(\xi) d\xi \\
 &=&   - \int_{-1}^1  \phi_{n+1}(\xi) L_{m-1}'(\xi) d\xi= 0.
\end{eqnarray*}  
 Since $L_{m-1}'\in \mathbb{P}_{m-2}(\tau)$, then 
\begin{equation}\label{eq: tmp4}
  \int_{-1}^1\phi_{n+1}(\xi) v(\xi)d\xi = 0, ~~~\forall \ v\in \mathbb{P}_{n-2}(\tau).
\end{equation}
In particular, choosing $v=1$ we have
\begin{equation*}
  \int_{-1}^1 \phi_{n+1}(\xi) d\xi = 0.
\end{equation*}
Since $\phi_{n+1}$ is continuous, and its average is zero over $(-1,1)$, therefore it must change signs at least once in $(-1,1)$. Let $\xi_1$, $\xi_2$, $\cdots$, $\xi_k$ be all points in $(-1,1)$ at which $\phi_{n+1}$ changes signs. We will show that $k=n-1$ by contradiction.\\

Suppose $k< n-1$. We choose $v(\xi) = (\xi-\xi_1)(\xi-\xi_2)\cdots(\xi-\xi_k)\in \mathbb{P}_{n-2}(\tau)$ so that  $\phi_{n+1}(\xi)v(\xi)$ does not change signs. The orthogonality \eqref{eq: tmp4} yields
\begin{equation}\label{eq: temp1}
  \int_{-1}^1 \phi_{n+1}(\xi) v(\xi) d\xi = 0.
\end{equation}
This contradicts \eqref{eq: temp1}. \\

Suppose $k> n-1$. Without loss of generality, we assume $-1<\xi_1<\cdots<\xi_k<1$ partitions $[-1,1]$ into $k+1$ subintervals, and $\hat\alpha\in(\xi_i,\xi_{i+1})$. On all $k$ noninterface subintervals, applying standard Rolle's theorem, we conclude that the derivative of $\phi_{n+1}(\xi)$ has at least one zero in each of these $k$ noninterface intervals. Hence, the weighted derivative $L_n(\xi) = \hat\beta\phi_{n+1}'(\xi)$ has at least $k$ zeros on  noninterface intervals.

On the interface subinterval $(\xi_i,\xi_{i+1})$, $\phi_{n+1}$ is not differentiable at the interior point $\hat\alpha$, then by the generalized Rolle's theorem (Lemma \ref{lemma: gen Rolle}), there exists a point $c$ such that one of $\phi_{n+1}'(c-)$ and $\phi_{n+1}'(c+)$ is non-negative, and the other is non-positive. It can be directly verified that 
\begin{equation*}
  L_n(c-) = \hat\beta(c-)\phi_{n+1}'(c-), ~~~~~
  L_n(c+) = \hat\beta(c+)\phi_{n+1}'(c+)
\end{equation*}
are also one of each, because $\hat\beta$ is strictly positive. Also, $L_n$ is a polynomial, thus continuous everywhere including at $c$.  Hence, $L_n(c) = 0$. That is, the polynomial $L_n$ has a zero in $(\xi_i,\xi_{i+1})$, which means $L_n(\xi)$ has at least $k+1 (>n)$ zeros on $(-1,1)$. This contradicts the first part of the theorem. 

In conclusion, $\phi_{n+1}$ has exactly $n-1$ roots on the open interval $(-1,1)$.
\end{proof}

Next we show the consistency of the generalized orthogonal polynomials with standard orthogonal polynomials.
\begin{lemma}
If the interface coincides with the boundary \emph{i.e.}, $\hat\alpha = \pm 1$, or if there is no jump of coefficient, \emph{i.e.}, $\beta^+ = \beta^-$, then $\{\phi_n\}$ and $\{L_n\}$ become standard Lobatto polynomial $\{\psi_n\}$ and Legendre polynomials $\{P_n\}$, respectively, up to a multiple constant.
\end{lemma}
\begin{proof}
Suppose $\hat\alpha = -1$. The weight function $w(\xi) = \frac{1}{\beta^+}$ becomes a constant. By the recurrence formula \eqref{eq: monic generalized Legendre poly}, it is easy to see that ${L}_n = c_{n}{P}_n$, where $c_n$ is a constant. By \eqref{eq: lobatto n} we have
\begin{equation*}
  \phi_n(\xi) = \int_{-1}^{\xi}\frac{1}{\beta^+}L_{n-1}(s)ds =  \frac{1}{\beta^+}c_{n-1}\int_{-1}^{\xi}{P}_{n-1}(s)ds = \frac{1}{\beta^+}c_{n-1}\psi_n(\xi),
\end{equation*}
for some constant $c_{n-1}$.\\

When $\hat\alpha = 1$, the argument is similar. When $\beta^+ = \beta^-$, the weight function $w(\xi) = \frac{1}{\beta^-}$ becomes a constant. The corresponding result can be obtained following a similar argument as above.
\end{proof}

We define a class of differential operators $D_x^j$ and integral operators $D_x^{-j}$, $j\geq1$:
\begin{equation}\label{eq: diff op}
(D_x^{1} v)\big|_{\tau_i} = (D_xv)\big|_{\tau_i} = v'(x), ~~  (D_x^{j}v)\big|_{\tau_i}=D_x(D_x^{j-1}v)\big|_{\tau_i}
\end{equation}
and $ D_x^{-j}: \tilde W_{\beta}^{m,q}(\Omega)\rightarrow \tilde W_{\beta}^{m,q}(\Omega), j\ge 1$  by 
\begin{equation}\label{eq: inte op}
     (D_x^{-1}v)\big|_{\tau_i}=\int_{x_{i-1}}^x v(x)dx,\ \
     (D_x^{-j}v)\big|_{\tau_i}=\int_{x_{i-1}}^xD_x^{-(j-1)}v(x)dx ,\ \ j\geq 2.
\end{equation} 
Next we prove an important inverse inequality for generalized polynomials. 
\begin{lemma}
\textbf{(Inverse Inequality)}~There exists a constant $C$, depending only on the polynomial degree $p$ such that 
\begin{equation}\label{eq: inverse}
|v|_{l,q,\tau_k}\leq C\rho h^{m-l+\frac{1}{q}-\frac{1}{r}}|v|_{m,r,\tau_k},~~~\forall v\in \tilde{\mathbb{P}}_p(\tau_k),
\end{equation}
where $1\leq q\leq \infty$, $1\leq r\leq \infty$, $0\leq m\leq l$, and  $\rho = \frac{\beta_{max}}{\beta_{min}}$.
\end{lemma}
\begin{proof} First we consider $q<\infty$, and $r<\infty$.
\begin{eqnarray}
|v|_{l,q,\tau_k}^q 
&=& \int_{\tau_k} |D_x^l v|^q dx\nonumber\\
%&=& \int_{\tau_k^+} |D^l v|^q dx + \int_{\tau_k^-} |D^l v|^q dx\nonumber\\
&=& \int_{\tau_k^+} \frac{1}{(\beta^+)^q}|D_x^{l-1} (\beta^+v')|^q dx+ \int_{\tau_k^-} \frac{1}{(\beta^-)^q}|D_x^{l-1} (\beta^-v')|^q dx \nonumber\\
%&=& \int_{\tau_k^+} \frac{1}{(\beta^+)^q}|D^{l-1}u|^q dx+ \int_{\tau_k^-} \frac{1}{(\beta^-)^q}|D^{l-1}u|^q dx\nonumber\\
&\leq& \frac{1}{(\beta_{min})^q} \int_{\tau_k} |D_x^{l-1} (\beta v')|^q dx \nonumber\\
&=& \frac{1}{(\beta_{min})^q} |\beta v'|_{l-1,q,\tau_k}^q\label{eq: 1st part}
\end{eqnarray}
Note that $\beta v'$ is a polynomial for all $v\in \tilde{\mathbb{P}}_p(\tau_k)$. In fact, $\beta v' \in {\mathbb{P}}_{p-1}(\tau_k)$. Standard inverse inequality \cite{1994BrennerScott} applied to $\beta v'$ yields
\begin{equation}\label{eq: 2nd part}
|\beta v'|_{l-1,q,\tau_k} \leq Ch^{m-l+\frac{1}{q}-\frac{1}{r}}|\beta v'|_{m-1,r,\tau_k}
\end{equation}
On the other hand,
\begin{eqnarray}
|\beta v'|_{m-1,r,\tau_k}^r
&=& \int_{\tau_k} |D_x^{m-1} (\beta v')|^r dx\nonumber\\
%&=& \int_{\tau_k^+} |D^{m-1} u|^r dx+\int_{\tau_k^-} |D^{m-1} u|^r dx\nonumber\\
&=& \int_{\tau_k^+}(\beta^+)^r |D_x^{m} v|^r dx+\int_{\tau_k^-} (\beta^-)^r|D_x^{m} v|^r dx\nonumber\\
&\leq& (\beta_{max})^r\int_{\tau_k}|D_x^{m} v|^r dx\nonumber\\
&=& (\beta_{max})^r|v|_{m,r,\tau_k}^r \label{eq: 3rd part}
\end{eqnarray}
Combining \eqref{eq: 1st part}, \eqref{eq: 2nd part}, and \eqref{eq: 3rd part} we have
\begin{equation}\label{eq: inverse p q less infty}
|v|_{l,q,\tau_k} \leq C\rho h^{m-l+\frac{1}{q}-\frac{1}{r}}|v|_{m,r,\tau_k}
\end{equation}

If $q=\infty$ (or $r=\infty$), we have
\begin{equation*}
|v|_{l,\infty,\tau_k} = \lim_{q\to\infty}|v|_{l,q,\tau_k} ~~~\forall  v\in \tilde{\mathbb{P}}_p(\tau_k).
\end{equation*}
Thus, the estimate \eqref{eq: inverse p q less infty} holds true.

\end{proof}

\section{Superconvergence Analysis}
In this section, we analyze the superconvergence property for the IFE method \eqref{eq: IFE method}. We first analyze the convergence estimates for interpolation. Then we discuss the superconvergence analysis for diffusion (only) interface problems \emph{i.e.}, $\gamma = c = 0$ in \eqref{eq: DE}. Finally, we consider the general elliptic interface problems, \emph{i.e.}, $\gamma \neq 0$, and $c\neq 0$.

\subsection{IFE Interpolation}
We consider the IFE interpolation using generalized Lobatto polynomials.    For any $u\in \tilde W_{\beta}^{m,q}(\Omega), m\ge 1$, we have the following Lobatto expansion of $u$ on noninterface elements $\tau_i$
\begin{equation}\label{eq: expansion on noninterface elem}
  u(x) |_{\tau_i}= \sum_{n = 0}^\infty u_{i,n}\psi_{i,n}(x),~~~i\neq k
\end{equation}
where
\begin{equation}\label{coeff:noninterface}
   u_{i,0} = u(x_{i-1}), \ \  u_{i,1} = u(x_{i}),\ \  u_{i,n}=\dfrac{\displaystyle\int_{\tau_i} u'(x)\psi_{i,n}'(x)dx}{\displaystyle\int_{\tau_i} \psi_{i,n}'(x)\psi_{i,n}'(x)dx}, \ \  n\ge 2.
  \end{equation}
On the interface element $\tau_k$,  since the flux $\beta u'$ is continuous, then it can be expanded by generalized Legendre polynomials $\{L_{k,n}\}$
\begin{equation*}\label{eq: expansion on interface elem 0}
  \beta u'(x) = \sum_{n = 0}^\infty u_{k,n}L_{k,n}(x){.}
\end{equation*}
Dividing by $\beta$ and then integrating on both sides yield the expansion for $u$
\begin{equation}\label{eq: expansion on interface elem}
  u(x)|_{\tau_k}
   =\sum_{n = 0}^\infty u_{k,n}\int_{x_{k-1}}^{x}\frac{1}{\beta(t)}L_{k,n}(t)dt =\sum_{n = 0}^\infty u_{k,n}\phi_{k,n}(x).
\end{equation}
  By the orthogonality  \eqref{eq: general orthogonal condition} and the properties of generized Lobatto polynomials in Theorem \ref{theo2:2}, we have
 \begin{equation}\label{coeff:interface}
    u_{k,0} = u(x_{k-1}), \ \ u_{k,1} = u(x_{k}),\ \ u_{k,n} = \frac{\langle u, \phi_{k,n}\rangle_{\tau_k}}{\langle \phi_{k,n},\phi_{k,n}\rangle_{\tau_k}},~~~n\geq 2,
\end{equation}
  where
% \begin{equation*}\label{eq: tmp2}
%  \langle \phi_{k,m},\phi_{k,n}\rangle_{\tau_k} = \frac{h_k}{2}\langle \phi_{m},\phi_{n}\rangle_{\tau}=   \int_{x_{k-1}}^{x_k}(\beta\phi_{k,m}'\phi_{k,n}')(x)dx= \delta_{mn}.
%\end{equation*}
 \begin{equation*}\label{eq: tmp2}
  \langle u,v\rangle_{\tau_k} =  \int_{x_{k-1}}^{x_k}\beta u'(x) v'(x)dx,~~~~\forall\  u,v\in \tilde W_{\beta}^{m,q}(\Omega).
\end{equation*}
Using the (generalized) Lobatto expansions \eqref{eq: expansion on noninterface elem} and \eqref{eq: expansion on interface elem} on noninterface and interface elements,
  we define the IFE interpolation $\mathcal{I}_h: \tilde W_{\beta}^{m,q}(\Omega) \to S_p{(\mathcal{T}_h)}$  as follows
\begin{equation}\label{eq: interpolation}
  (\mathcal{I}_hu)|_{\tau_i} =
  \left\{
    \begin{array}{ll}
      \sum\limits_{n=0}^p u_{i,n}\psi_{i,n}(x), & \text{if}~ i\neq k \vspace{1mm}\\
      \sum\limits_{n=0}^p u_{i,n}\phi_{i,n}(x), & \text{if}~ i = k.
    \end{array}
  \right.
\end{equation}
   %The IFE interpolation $\mathcal{I}_hu$ plays an important role in our superconvergence analysis.
   
%To study the approximation properties of $\mathcal{I}_hu$,  we define a class of integral operators $  D_x^{-j}: \tilde W_{\beta}^{m,q}(\Omega)\rightarrow \tilde W_{\beta}^{m,q}(\Omega), j\ge 1$  by 
%\[
%     (D_x^{-1}v)\big|_{\tau_i}=\int_{x_{i-1}}^x v(x)dx,\ \
%     (D_x^{-j}v)\big|_{\tau_i}=\int_{x_{i-1}}^xD_x^{-(j-1)}v(x)dx ,\ \ j\geq 2.
%\]

  \begin{lemma}
       There holds for all $j\le n-2$
 \begin{equation}\label{inter:1}
          D_x^{-j}\phi_{k,n}(x_{k-1})= D_x^{-j}\phi_{k,n}(x_{k})= 0
\end{equation}
where $D_x^{-j}$ is the integral operator defined in \eqref{eq: inte op}.
 \end{lemma}
 \begin{proof}
      Note from \eqref{eq: tmp4} and Theorem \ref{theo2:2} that
 \begin{equation}\label{iin}
 \phi_{k,n}(x_{k-1})=\phi_{k,n}(x_k)=0,\ \ \  \int_{x_{k-1}}^{x_k}\phi_{k,n} (x)v(x)dx=0,\ \ \forall\  v\in \mathbb P_{n-3}(\tau_k).
  \end{equation}
     Choosing $v=1$ in the above equation, we  immediately  obtain
 \[
   D_x^{-1}\phi_{k,n}(x_{k})=\int_{x_{k-1}}^{x_k}\phi_{k,n}(x)dx=0=
   D_x^{-1}\phi_{k,n}(x_{k-1}),\ \ \forall \ n\ge 3.
\]
  Moveover, noticing that $D_x^{-1}v\in \mathbb P_{n-3}(\tau_k)$ for all $v\in\mathbb P_{n-4}(\tau_k)$,   we have,  from \eqref{iin} and the integration by parts,
 \[
    \int_{\tau_k}D_x^{-1}\phi_{k,n}(x)v(x)dx=-\int_{\tau_k}\phi_{k,n}(x)D_x^{-1}v(x)dx=0,\ \ \forall \ v\in\mathbb P_{n-4}(\tau_k).
\]
   In other words,  $D_x^{-1}\phi_{k,n}$ shares the same properties  of $\phi_{k,n}$,  \emph{i.e.},
 \[
    D_x^{-1}\phi_{k,n}(x_{k})=
        D_x^{-1}\phi_{k,n}(x_{k-1})= 0,\ \ \  \int_{x_{k-1}}^{x_k}D_x^{-1}\phi_{k,n}(x)v(x)dx=0,\ \ v\in \mathbb P_{n-4}(\tau_k).
 \]
 By recursion,  there holds for all $j\le n-3$
\[
     D_x^{-j}\phi_{k,n}(x_{k})=
        D_x^{-j}\phi_{k,n}(x_{k-1})= 0,\ \ \  \int_{x_{k-1}}^{x_k}D_x^{-j}\phi_{k,n}(x)v(x)dx=0,\ \ v\in \mathbb P_{n-3-j}(\tau_k),
\]
%  Choosing $v=1$ in the above equation, we obtain for all $n\ge 4$,
%\[
%    D_x^{-2}\phi_{k,n}(x_{k})=\int_{x_{k-1}}^{x_{k}}D_x^{-1}\phi_{k,n}=0=D_x^{-2}\phi_{k,n}(x_{k-1}) .
% \]
%   Similarly,  using the integration by parts and \eqref{iin},
% \[
 %      \int_{\tau_k}D_x^{-2}\phi_{k,n}(x)v(x)dx=-\int_{\tau_k}D_x^{-1}\phi_{k,n}(x)D_x^{-1}v(x)dx=0,\ \ \forall \ v\in\mathbb P_{n-5}.
 %\]
 %
   %  \[
   %  D_x^{-(j-1)}\phi_{k,n}(x_{k})=0= D_x^{-(j-1)}\phi_{k,n}(x_{k-1}),\ \ \  D_x^{-(j-1)}\phi_{k,n}\bot\mathbb P_{n-2-j}(\tau_k),
  % \]
  % we obtain
 % \[
  %    D_x^{-j}\phi_{k,n}(x_{k})= D_x^{-j}\phi_{k,n}(x_{k-1})=0, \ \ \  D_x^{-j}\phi_{k,n}\bot\mathbb P_{n-3-j}(\tau_k),   \ \ j\le n-3.
   % \]
    which yields
   \[
        D_x^{-(j+1)}\phi_{k,n}(x_k)=\int_{x_{k-1}}^{x_k}D_x^{-j}\phi_{k,n}(x)dx=0=D_x^{-(j+1)}\phi_{k,n}(x_{k-1}),\ \  j\le n-3.
   \]
     This finishes our proof.
  \end{proof}

   Now we are ready to show the approximation properties of the IFE interpolation  $\mathcal{I}_hu$.

\begin{lemma}
%Let $\mathcal{T}_h = \{\tau_i\}_{i = 1}^N$ be a mesh of $\Omega$ such that $\tau_k$ contains the interface point $\alpha$.
Assume that $u\in \tilde W_{\beta}^{p+2,\infty}(\Omega)$, and $\mathcal I_hu$ is the IFE interpolation of $u$ defined by \eqref{eq: interpolation}. The following orthogonality and approximation properties hold true. 
  \begin{enumerate}
  \item Orthogonality:
\begin{equation}\label{ortho:1}
      \int_{\tau_i}\beta(u-\mathcal I_hu)'v' dx=0,\ \ \forall \ v\in S_p(\mathcal{T}_h),\ i=1,\ldots,N.
 \end{equation}
  \item Superconvergence on noninterface elements $\tau_i$, $i\ne k$: There exists a constant $C$ depending only on the polynomial degree $p$ such that 
%  $\mathcal I_hu$ is superconvergent at the interior roots of the Lobatto polynomials $\psi_{i,p+1}$ for $p\ge 2$, and the derived value of $\mathcal I_hu$ is superconvergent at roots of the Legendre polynomial $P_{i,p}$. That is
\begin{equation}\label{approx:super1}
  |(u-{\cal I}_hu) (l_{im}) | \le C h^{p+2}|u|_{p+2,\infty,\tau_i},~~~i\ne k,
\end{equation}
\begin{equation}\label{approx:super2}
  |(u'-({\cal I}_hu)') (g_{in})| \leq Ch^{p+1}|u|_{p+2,\infty,\tau_i},~~~~i\neq k,
\end{equation}
where $l_{im}$, $m=1,\cdots, p-1$ are interior roots of $\psi_{i,p+1}$ on $\tau_i$, and $g_{in}$, $n=1,\cdots, p$ are roots of $P_{i,p}$ on $\tau_i$.
 \item Superconvergence on interface element $\tau_k$: There exists a constant $C$ depending only on the polynomial degree $p$ and the ratio of coefficient $\rho$ such that 
% $\mathcal I_hu$ is superconvergent at the interior roots of the generalized Lobatto polynomials $\phi_{k,p+1}$, and  the weighted derivative of $\mathcal I_hu$ is superconvergent at roots of the generalized Legendre polynomial  $L_{k,p}$. That is
\begin{equation}\label{approx:super4}
  |(u-{\cal I}_hu) (l_{km}) | \le C h^{p+2}|u|_{p+2,\infty,\tau_k},
\end{equation}
\begin{equation}\label{approx:super3}
  |(\beta u'-(\beta{\cal I}_hu)') (g_{kn})| \leq Ch^{p+1}|u|_{p+2,\infty,\tau_k},
\end{equation}
where $l_{km}$, $m=1,\cdots, p-1$ are interior roots of $\phi_{k,p+1}$ on $\tau_k$, and $g_{kn}$, $n=1,\cdots, p$ are roots of $L_{k,p}$ on $\tau_k$.
\end{enumerate}
 \end{lemma}

\begin{proof}
   By \eqref{eq: expansion on noninterface elem}, \eqref{eq: expansion on interface elem} and \eqref{eq: interpolation}, we have
 \begin{equation}\label{eq: interpolation error}
  (u-\mathcal{I}_hu)|_{\tau_i} =
  \left\{
    \begin{array}{ll}
      \sum\limits_{n=p+1}^\infty u_{i,n}\psi_{i,n}(x), & \text{if}~ i\neq k, \vspace{1mm}\\
      \sum\limits_{n=p+1}^\infty u_{i,n}\phi_{i,n}(x), & \text{if}~ i= k.
    \end{array}
  \right.
\end{equation}
   Then  \eqref{ortho:1} follows from the orthogonal properties of  (generalized) Lobatto polynomials.

On each noninterface element $\tau_i, i\neq k$, we have from \eqref{coeff:noninterface}
  \begin{eqnarray}
  u_{i,n}
   &=& \frac{h_i}{2n-1}\int_{\tau_i} u'(x)\psi_{i,n}'(x)dx
   = \frac{2}{2n-1}\int_{-1}^1 \frac{d u(\xi)}{d\xi}P_{n-1}(\xi)d\xi \nonumber\\
   &=& \frac{2}{2n-1}\frac{1}{{(n-1)!}2^{n-1}}\int_{-1}^1 \frac{d u(\xi)}{d\xi} \frac{d^{n-1}}{d\xi^{n-1}}(\xi^2-1)^{n-1}d\xi \nonumber\\
   &=& \frac{2}{2n-1}\frac{(-1)^{j-1}}{{(n-1)!}2^{n-1}}\int_{-1}^1 \frac{d^{j} u(\xi)}{d\xi^j}\frac{d^{n-j}}{d\xi^{n-j}}(\xi^2-1)^{n-1}d\xi,\ \ j\le n\nonumber.
   % &=& \frac{2}{2n-1}C_{n,j}\left(\frac{h_i}{2}\right)^{n-1}\int_{\tau_i} \frac{d^{n}u(x)}{dx^n}\psi_{i2}(x)dx \nonumber.
\end{eqnarray}
Since
\[
   \frac{d^{j} u(\xi)}{d\xi^j}=\left(\frac{h_i}{2}\right)^j\frac{d^{j} u(x)}{dx^j},
  \]
  then let $j=n$, we have 
\begin{equation}\label{eq: bound for uim}
  |u_{i,n}|\leq C_{n}h^n|u|_{n,\infty,\tau_i},
\end{equation}
where $C_{n}$ is a positive constant depending only on $n$. By \eqref{eq: interpolation error} and \eqref{eq: bound for uim} we can show \eqref{approx:super1} as follows
\begin{eqnarray*}
 (u-\mathcal{I}_hu)(l_{im}) &=&  \sum\limits_{n=p+1}^\infty u_{i,n}\psi_{i,n}(l_{im}) \leq |u_{i,p+2}||\psi_{i,p+2}(l_{im})| + O(h^{p+3}) \\
 &\leq& C_p h^{p+2}|u|_{p+2,\infty,\tau_i},
\end{eqnarray*}
where $C_p$ depends only on the polynomial degree $p$.

On the interface element $\tau_k$, by \eqref{coeff:interface},
 \begin{eqnarray*}
       u_{k,n}&=&\frac{1}{\langle \phi_n,\phi_n\rangle_{\tau}}\frac{h_k}{2}\int_{\tau_k}(\beta u')(x)\phi'_{k,n}(x)dx\\
              &=&  \frac{(-1)^{j-1}}{\langle \phi_n,\phi_n\rangle_{\tau}}\frac{h_k}{2}
 \int_{x_{k-1}}^{x_{k}}(\beta u')^{(j+1)}(x)D_x^{-j}\phi_{k,n}(x)dx,\ \ j\le n-2.
   \end{eqnarray*}
      Here in the last step, we have used the integration by parts and \eqref{inter:1}. We let $j=n-2$, and use the estimate $ \|D_x^{-1}v\|_{0,\infty}\le h\|v\|_{0,\infty}$ to obtain
 \begin{equation}\label{eq: bound for uim 2}
  |u_{k,n}|\leq  C_{n}h^{n}|\beta u'|_{n-1,\infty,\tau_k}\leq C_{n,\rho}h^{n}|u|_{n,\infty,\tau_k},\end{equation}
where $C_{n,\rho}$ depends on $n$ and the coefficient ratio $\rho$.
   Then \eqref{approx:super4} follow from \eqref{eq: interpolation error} and \eqref{eq: bound for uim 2}
 \begin{eqnarray*}
 (u-\mathcal{I}_hu)(l_{km}) &=&  \sum\limits_{n=p+1}^\infty u_{k,n}\phi_{i,n}(l_{km}) \leq |u_{k,p+2}||\phi_{k,p+2}(l_{km})| + O(h^{p+3}) \\
 &\leq& C_{p,\rho} h^{p+2}|u|_{p+2,\infty,\tau_k},
\end{eqnarray*} 
 where $C_{p,\rho}$ depends only on the polynomial degree $p$ and coefficient ratio $\rho$. \\
 
For derivatives, we note that
  \begin{eqnarray*}
  (u'-(\mathcal{I}_hu)')|_{\tau_i} &=& \frac{2}{h_i}\sum\limits_{n=p}^\infty u_{i,n}P_{i,n}(x),~~~ \text{if}~ i\neq k,\label{eq: interpolation error derivative}\\
  (\beta u'-(\beta \mathcal{I}_hu)')|_{\tau_k} &=&  \frac{2}{h_k}  \sum\limits_{n=p}^\infty u_{k,n}L_{k,n}(x).\label{eq: interpolation error flux}
\end{eqnarray*}
Then \eqref{approx:super2} and \eqref{approx:super3} follow from \eqref{eq: bound for uim}-\eqref{eq: bound for uim 2}.  The proof is complete.
\end{proof}

\subsection{Superconvergence for diffusion interface problems}

We first consider the diffusion interface problem, \emph{i.e.}, $\gamma = c = 0$ in \eqref{eq: DE}. Assume that $u_h\in S_p(\mathcal{T}_h)$ is the IFE solution of
\begin{equation}\label{eq: IFE method Diffusion}
  a(u_h,v_h) := (\beta u_h',v_h') = (f,v_h),~~~\forall v_h\in S_p(\mathcal{T}_h).
\end{equation}
By the Poincar\'{e} inequality, and the orthogonality \eqref{ortho:1}, we have
\begin{eqnarray*}
  \|\mathcal{I}_hu-u_h\|_{1}^{2}\leq  C a(\mathcal{I}_hu-u_h,\mathcal{I}_hu-u_h)\leq C a(\mathcal{I}_hu-u,\mathcal{I}_hu-u_h) = 0.
\end{eqnarray*}
Hence, $u_h = \mathcal{I}_hu$. That means $u_h$ inherits all superconvergent properties \eqref{approx:super1} - \eqref{approx:super3} of ${\cal I}_hu$.
We summarize these results in the following theorem.
\begin{theorem}
Let $\mathcal{T}_h = \{\tau_i\}_{i=1}^N$ be a mesh of $\Omega$ such that the interface $\alpha\in \tau_k$.  Let $u_h\in S_p(\mathcal{T}_h)$ be the IFE solution of \eqref{eq: IFE method Diffusion} where $p\ge 2$, and $u\in \tilde W_\beta^{p+2,\infty}(\Omega)$ be the exact solution of \eqref{eq: DE} - \eqref{eq: jump condition}. Then we have the following results.
\begin{itemize}
  \item $u_h$ is exact at the mesh points, \emph{i.e.},
\begin{equation}\label{eq: exact on nodes}
  (u - u_h) (x_i) = 0,~~~~\forall~ i = 0,1,\cdots, N.
\end{equation}
  \item On every noninterface element $\tau_i$, $i\ne k$, $u_h$ is superconvergent at roots of Lobatto polynomial $\psi_{i,p+1}$, and the derivative $u_h'$ is superconvergent at roots of Legendre polynomial $P_{i,p}$. That is, there exists a constant $C$ depending only on polynomial degree $p$ such that 
\begin{equation}\label{eq: supercvg Lobatto Gauss}
  (u - u_h) (l_{im}) = Ch^{p+2}|u|_{p+2,\infty},~~
  (u' - u_h') (g_{in}) = Ch^{p+1}|u|_{p+2,\infty}.
\end{equation}
  \item On the interface element $\tau_k$, $u_h$ is superconvergent at roots of generalized Lobatto polynomial $\phi_{k,p+1}$, and the flux $\beta u_h'$ is superconvergent at roots of generalized Legendre polynomial $L_{k,p}$. That is, there exists a constant $C$ depending only on polynomial degree $p$ and the ratio of coefficient jump $\rho$ such that
\begin{equation}\label{eq: supercvg Lobatto Gauss inter}
  (u - u_h) (l_{km}) = Ch^{p+2}|u|_{p+2,\infty},~~
  (\beta u' - \beta u_h') (g_{kn}) = Ch^{p+1}|u|_{p+2,\infty}.
\end{equation}
\end{itemize}
\end{theorem}

\subsection{Superconvergence for general elliptic interface problems}

We consider the general second-order elliptic interface problem. As the standard finite element approximation, we cannot expect $u_h$ is exact at the mesh points. However, we may establish similar superconvergence results as the counterpart finite element methods by using  the superconvergence analysis tool. To this end, we will need to construct a special function  $\omega$.  Define 
\[
  \tilde S_p(\mathcal{T}_h) := \{v\in H^1(\Omega): v|_{\tau_i} \in \mathbb{P}_p(\tau_i), ~i\ne k, ~\text{and}~ v|_{\tau_k} \in \tilde{\mathbb{P}}_p(\tau_k), v(a)=0\}.
\]
  Let $\omega\in\tilde S_p(\mathcal{T}_h)$ be a  function satisfying 
 \begin{equation}\label{corr:w}
   (\beta\omega',v')=(\gamma (u-{\cal I}_hu),v'),\ \ \forall v\in \tilde S_p({\cal T}_h).  
\end{equation}
   Apparently, the Lax-Milgram theory assures the existence and uniqueness of the solution $\omega$. Moreover, we have the following estimate for $\omega$.

   \begin{lemma}  Let $u\in \tilde W^{p+1,\infty}_{\beta}(\Omega)$ and  $\omega\in  \tilde S_p(\mathcal{T}_h)$ be the special function defined by \eqref{corr:w}.  Then for all $p\ge 2$, 
 \begin{equation}\label{esti:w}
    \|\omega\|_{0,\infty}\le C h^{p+2}\|u\|_{p+1,\infty}, 
\end{equation}
 where $C$ is a positive constant  depending only on the coefficients $\beta$ and $\gamma$. 
 \end{lemma}
\begin{proof}  In each element $\tau_i$, we assume that $\omega$ has the following expansion
\begin{equation}\label{w expansion}
 \omega|_{\tau_i} =
  \left\{
    \begin{array}{ll}
      \sum\limits_{n=2}^p c_{i,n}\psi_{i,n}(x)+\omega(x_{i-1})\psi_{i,0}(x)+\omega(x_i)\psi_{i,1}(x), & \text{if}~ i\neq k, \vspace{1mm}\\
      \sum\limits_{n=2}^p c_{i,n}\phi_{i,n}(x)+\omega(x_{i-1})\phi_{i,0}(x)+\omega(x_i)\phi_{i,1}(x), & \text{if}~ i= k. 
    \end{array}
  \right.
\end{equation}
   By choosing $v=\psi_{i,n}$  and $v=\phi_{k,n}$  in \eqref{corr:w}, where $2\le n\le p$, we can find
 \[ 
c_{i,n}=
  \left\{
    \begin{array}{ll}
      \frac{2n-1}{2}\int_{\tau_i}\frac{\gamma}{\beta} (u-{\cal I}_hu)(x) P_{i,n-1}(x) dx, & \text{if}~ i\neq k, \vspace{1mm}\\
    \dfrac{1}{\langle \phi_n,\phi_n\rangle_{\tau}}\int_{\tau_i}\frac{\gamma}{\beta} (u-{\cal I}_hu)(x) L_{i,n-1}(x) dx , & \text{if}~ i= k. 
    \end{array}
  \right.
 \] 
   Apparently, by the standard approximation theory, 
\begin{equation}\label{esti:c}
    | c_{i,n}|\le C h\|u-{\cal I}_hu\|_{0,\infty}\le C h^{p+2}\|u\|_{p+1,\infty}. 
\end{equation}
   Here  the constant $C$ depends only on the coefficients $\beta$ and $\gamma$. 
    Similarly, we separately choose  $v'=P_{i,0}=1, i\neq k$ and  $v'=\phi'_{k,1}$ in \eqref{corr:w}  to 
     obtain 
 \[
     \omega(x_i)-\omega(x_{i-1})=\int_{\tau_i}\frac{\gamma}{\beta} (u-{\cal I}_hu)(x)dx,   \ \forall i. 
 \]
  %  and 
%\[
%  c_{i,n}=\frac{2n-1}{2}\int_{\tau_i}\frac{\gamma}{\beta} (u-{\cal I}_hu)(x) P_{i,n-1}(x) dx,\ \ n\ge 2. 
% \]
     In light of  \eqref{eq: interpolation error} and the orthogonal properties of (generalized) Lobotto polynomials, we know that $u-{\cal I}_hu$ is orthogonal to polynomials of degree no more than ${p-2}$. Then  for $p\ge 2$
\[
     \omega(x_i)-\omega(x_{i-1})=0,\ \ i\neq k. 
\]
       Since $\omega(x_0)=\omega(a)=0$,  we have 
\[ 
\omega(x_i) =
  \left\{
    \begin{array}{ll}
    0, & \text{if}~ i\le k-1, \vspace{1mm}\\
   \int_{\tau_k}\frac{\gamma}{\beta} (u-{\cal I}_hu)(x) dx, & \text{if}~ i \ge k. 
    \end{array}
  \right.
\]
   Consequently, 
\begin{equation}\label{esti wxi}
     |\omega(x_i)|\le C  h\|u-{\cal I}_hu\|_{0,\infty}\le C h^{p+2}\|u\|_{p+1,\infty},\ \ \forall i. 
\end{equation}
   Then the estimate \eqref{esti:w} follows from \eqref{w expansion}, \eqref{esti:c}, and \eqref{esti wxi}.
\end{proof}

    Now we are ready to show the superconvergence for general elliptic interface problems.

 \begin{theorem}\label{theo:11}
Let $\mathcal{T}_h = \{\tau_i\}_{i=1}^N$ be an partition of $\Omega$ such that the interface $\alpha\in \tau_k$.  Let $u_h\in S_p(\mathcal{T}_h)$ be the IFE solution of \eqref{eq: IFE method} where $p\ge 2$, and $u\in \tilde W_\beta^{p+2,\infty}(\Omega)$ be the exact solution of \eqref{eq: DE} - \eqref{eq: jump condition}. Then we have the following superconvergence results.  
\begin{itemize}
%  \item At the mesh points, $u_h$ is superconvergent when $p\ge 3$, \emph{i.e.},
%\begin{equation}\label{eq: supercvg nodes General}
%  (u - u_h) (x_i) = Ch^{p+2}\|u\|_{p+2,\infty},~~ i = 0,1,\cdots, N.
%\end{equation}
  \item There exists a constant $C$, depending on $p$, $\rho$, $\gamma$, $c$ such that the following estimate holds true on every noninterface element $\tau_i$, $i\ne k$.
  \begin{equation}\label{eq: supercvg Lobatto Gauss General}
  (u - u_h) (l_{im}) = Ch^{p+2}\|u\|_{p+2,\infty},~~
  (u' - u_h') (g_{in}) = Ch^{p+1}\|u\|_{p+2,\infty},
\end{equation}
  where $l_{im}$, $m=0,1,cdots,p$ are roots of $\psi_{i,p+1}$, including the mesh points, and $g_{in}$, $n=1,2,\cdots, n$ are roots of $P_{i,p}$.
  \item There exists a constant $C$, depending on $p$, $\rho$, $\gamma$, $c$ such that the following estimate holds true on the interface element $\tau_k$.
  \begin{equation}\label{eq: supercvg Lobatto Gauss inter General}
 (u - u_h) (l_{km}) = Ch^{p+2}\|u\|_{p+2,\infty},~~
  (\beta u' - \beta u_h') (g_{kn}) = Ch^{p+1}\|u\|_{p+2,\infty},
\end{equation}
where $l_{km}$, $m=0,1,\cdots,p$ are roots of $\phi_{i,p+1}$, including the mesh points, and $g_{kn}$, $n=1,2,\cdots, n$ are roots of $L_{k,p}$.
\end{itemize}
\end{theorem}
\begin{proof}  First,  let 
\[
   u_I={\cal I}_hu+\omega, 
\]
  where $\omega$ is defined by \eqref{corr:w}.  By  \eqref{eq: IFE method}  and the coercivity of the bilinear form of the IFE method,  we have
\begin{eqnarray*}
   \|u_h-u_I\|_{1}^{2}
& \leq & C a(u_h-u_I,u_h-u_I)
 =C a(u-u_I,u_h-u_I). 
 % &=&  (\beta (u-u_I)', (u_h-u_I)')+ (\gamma  (u-u_I),(u_h-u_I)') + (c(u-u_I),u_h-u_I)
\end{eqnarray*}
  By \eqref {ortho:1} and \eqref{corr:w}, we have 
\begin{eqnarray*}
    |a(u-u_I,v)|&=&|(c(u-{\cal I}_hu),v)-(\gamma \omega, v')-(c \omega,v)|\\
        &=&|-(cD_x^{-1}(u-{\cal I}_hu),v')-(\gamma \omega, v')-(c \omega,v)|\\
        &\le & C \left( h \|u-{\cal I}_hu\|_{0,\infty}+\| \omega\|_{0,\infty}\right)\|v\|_{1}, \ \ \forall v\in S_p({\cal T}_h), 
\end{eqnarray*}
   where in the second step, we have used the integration by parts, and the fact that 
 \[
    D_x^{-1}(u-{\cal I}_hu)(x_{i})= D_x^{-1}(u-{\cal I}_hu)(x_{i-1})=0. 
 \]
    Consequently,  
 \[
      \|u_h-u_I\|_{1}\le C \left( h \|u-{\cal I}_hu\|_{0,\infty}+\| \omega\|_{0,\infty}\right) \le Ch^{p+2}\|u\|_{p+1,\infty}. 
  \]
    Noticing  that $(u_h-u_I)(a)=0$,  we have 
 \[
     (u_h-u_I)(x)=\int_{a}^x  (u_h-u_I)'(x) dx,  
 \]
   which yields 
 \[
      \|u_h-u_I\|_{0,\infty}\le C |u_h-u_I|_1\le Ch^{p+2}\|u\|_{p+1,\infty}, 
 \]
   and thus, 
  \[
      \|u_h-{\cal I}_hu\|_{0,\infty}\le  Ch^{p+2}\|u\|_{p+1,\infty}+\|\omega\|_{0,\infty}\le C h^{p+2}\|u\|_{p+1,\infty}. 
 \]
    Since $u_h-{\cal I}_hu\in S_p( {\cal T}_h)$, the inverse inequality holds. Then 
  \[
      |u_h-{\cal I}_hu|_{1,\infty}\le  C h^{-1}  \|u_h-{\cal I}_hu\|_{0,\infty}\le C h^{p+1}\|u\|_{p+1,\infty}. 
 \]  
      Then \eqref{eq: supercvg Lobatto Gauss General} and \eqref{eq: supercvg Lobatto Gauss inter General} follow from \eqref{approx:super1}- \eqref{approx:super2}. 
\end{proof}

%\begin{remark}
%The assumption $p\ge 2$ in Theorems 3.1 and 3.2 rules out the linear element case $p=1$ only because the convergent rate $O(h^{p+2})$ at the Lobatto (generalized Lobatto) points are not valid for linear element, since the best rate is $h^{2p}$, which is $h^2$ in the linear case. All other results are trivially extended to the linear element case.
%\end{remark}

\begin{remark}
    As we may recall,  the convergence rates $O(h^{p+2})$ at the Lobatto (generalized Lobatto) points and 
     $O(h^{p+1})$ at the Gauss (generalized Lobatto) points are the same as these of the counterpart FEM. 
    While,  as for the convergence rate at mesh nodes,  the order $O(h^{p+2})$ in the Theorem \ref{theo:11} is lower than  that  of the FEM for $p\ge 3$, which is  $O(h^{2p})$.  Nevertheless, our numerical experiments demonstrate that the convergence rate at mesh points sometimes might be even higher than $O(h^{p+2})$.
\end{remark}

\begin{remark}
For problems with multiple interface points, the analytical results in Theorem 4.1 and Theorem 4.2 are still true. Example 5.2 provides a numerical evidence for this scenario.
\end{remark}

\begin{remark}
In general, there is no superconvergence at the interface point, because the IFE method treats the interface as an interior point. Even if there is no coefficient jump, the IFE method (becomes standard FE method) has no superconvergence behavior at a random interior point, unless it coincides with Lobatto or Gauss points.
\end{remark}

\section{Numerical Experiments}
In this section, we present some numerical experiments to demonstrate the superconvergence of IFE methods. 

We use a family of uniform mesh $\{\mathcal{T}_h\}, h>0$ where $h$ denotes the mesh size. We will test linear ($p$=1), quadratic ($p$=2), and cubic ($p$=3) IFE approximation. In the following experiments, we always start from a mesh consisting of eight elements. Due to the finite machine precision, we choose different sets of meshes for different polynomial degrees $p$. The convergence rate is calculated using linear regression of the errors.

We compute the error $e_h=u_h-u$ in the following norms
\begin{eqnarray*}
  &\|e_h\|_{node} = \max\limits_{x\in{\{x_i\}}}|u_h(x)-u(x)|,
  &\|e_h\|_{L^\infty} = \max\limits_{x\in\Omega}|u_h(x)-u(x)|, \\
  &\|e_h\|_{Lob} = \max\limits_{x\in\{l_{ip}\}}|u_h(x)-u(x)|,
  &\|\beta e_h'\|_{Gau} = \max\limits_{x\in\{g_{ip}\}}|\beta u_h'(x)-\beta u'(x)|, \\
  &\|e_h\|_{L^2} = \left(\displaystyle\int_\Omega|u_h-u|^2dx\right)^{\frac{1}{2}},
  &|e_h|_{H^1} = \left(\int_\Omega|u_h'-u'|^2dx\right)^{\frac{1}{2}}.
\end{eqnarray*}
Here, $\|e_h\|_{node}$ denotes the maximum error over all the nodes (mesh points). $\|e_h\|_{L^\infty}$ is the infinity norm over the whole domain $\Omega$. To compute it, we select eight uniformly distributed points on each non-interface element, and select 10 points in each sub-element of an interface element. Among all these sample points, we compute the largest discrepancy from the exact solution. $\|\beta e_h'\|_{Gau}$ is the maximum error of flux over all (generalized) Legendre points. $\|e_h\|_{Lob}$ is maximum error over all (generalized) Lobatto points, respectively. $\|e_h\|_{L^2}$ and $|e_h|_{H^1}$ are the standard Sobolev $L^2$- and semi-$H^1$- norms.

\begin{example} \textbf(One interface point)
In this example, we consider an interface problem with one interface point. We use the following example as the exact solution 
\begin{equation}\label{eq: IFE ex1}
  u(x) =
  \left\{
    \begin{array}{ll}
      \dfrac{1}{\beta^-}\cos(x), & \text{if}~~x \in [0,\alpha), \\
      \dfrac{1}{\beta^+}\cos(x) + \left(\dfrac{1}{\beta^-} - \dfrac{1}{\beta^+}\right)\cos(\alpha), & \text{if}~~x \in (\alpha,1]. \\
    \end{array}
  \right.
\end{equation}
It is easy to verify that
\begin{equation*}
  \jump{u(\alpha)} = 0,~~~
  \bigjump{\beta u^{(j)}(\alpha)} = 0, ~~\forall j \geq 1.
\end{equation*}
\end{example}

We consider the general elliptic interface problem, and choose the coefficient $(\beta^-,\beta^+) = (1,5)$, $\gamma=1$, $c = 10$, and the interface $\alpha = \pi/6$. Errors of the IFE solution of degree $p=1,2,3$ in the aforementioned norms are reported in Tables \ref{table: general elliptic P1-IFE-1-5}, \ref{table: general elliptic P2-IFE-1-5}, and \ref{table: general elliptic P3-IFE-1-5}, respectively.  At (generalized) Legendre-Gauss points and (generalized) Lobatto points (for $p$ = 2,3), the convergence rates are $O(h^{p+1})$ and $O(h^{p+2})$, respectively. At mesh points, the IFE solutions $u_h$ demonstrate a superconvergence order of at least $O(h^{p+2})$ for $p=2,3$, compared to the rate $O(h^{p+1})$ in the infinity norm $\|\cdot\|_{L^\infty}$. These data indicate that at these special points, IFE solution are super-close to the exact solution, and the convergence rates are one order higher than optimal rate. Moreover, the convergence rates are $O(h^{p+1})$ and $O(h^p)$ in $\|\cdot\|_{L^2}$ and $|\cdot|_{H^1}$ norm, which is consistent with the diffusion interface problem in \cite{2009AdjeridLin}.

\begin{table}[thb]
\begin{center}
\begin{small}
\begin{tabular}{|r|c|c|c|c|c|}
\hline
$1/h$ & $\|e_h\|_{node}$ & $\|e_h\|_{L^\infty}$ & $\|\beta e_h'\|_{Gau}$ &$\|e_h\|_{L^2}$ & $\|e_h\|_{H^1}$\\
\hline
    8 & 5.71e-05 & 1.92e-03 & 1.07e-03 & 9.97e-04 & 2.51e-02 \\
  16 & 1.43e-05 & 4.81e-04 & 2.75e-04 & 2.48e-04 & 1.24e-02 \\
  32 & 3.25e-06 & 1.20e-04 & 6.98e-05 & 6.21e-05 & 6.26e-03 \\
  64 & 5.44e-07 & 3.01e-05 & 1.75e-05 & 1.56e-05 & 3.14e-03 \\
128 & 2.07e-07 & 7.53e-06 & 4.40e-06 & 3.91e-06 & 1.58e-03 \\
256 & 5.16e-08 & 1.88e-06 & 1.10e-06 & 9.78e-07 & 7.88e-04 \\
512 & 1.29e-08 & 4.71e-07 & 2.76e-07 & 2.44e-07 & 3.94e-04 \\
\hline
rate & 2.02        & 1.99        & 1.99        & 2.00         & 1.00  \\
\hline\end{tabular}
\caption{Error of $P_1$ IFE Solution with $\beta=[1,5]$, $\alpha = \pi/6$, $\gamma=1$, $c=1$.}
\label{table: general elliptic P1-IFE-1-5}
\end{small}
\end{center}
\end{table}

\begin{table}[thb]
\begin{center}
\begin{small}
\begin{tabular}{|r|c|c|c|c|c|c|}
\hline
$1/h$ & $\|e_h\|_{node}$ & $\|e_h\|_{L^\infty}$ &$\|e_h\|_{Lob}$ & $\|\beta e_h'\|_{Gau}$ &$\|e_h\|_{L^2}$ & $\|e_h\|_{H^1}$\\
\hline
    8 & 3.69e-08 & 6.69e-06 & 2.98e-07 & 9.74e-06 & 2.51e-06 & 1.31e-04 \\
  16 & 5.22e-09 & 8.85e-07 & 1.63e-08 & 1.25e-06 & 3.17e-07 & 3.33e-05 \\
  24 & 1.17e-09 & 2.67e-07 & 3.13e-09 & 3.67e-07 & 9.47e-08 & 1.48e-05 \\
  32 & 1.77e-10 & 1.14e-07 & 1.20e-09 & 1.55e-07 & 3.97e-08 & 8.25e-06 \\
  40 & 3.60e-11 & 5.88e-08 & 5.57e-10 & 7.98e-08 & 2.07e-08 & 5.38e-06 \\
  48 & 2.42e-11 & 3.54e-08 & 2.65e-10 & 4.67e-08 & 1.21e-08 & 3.76e-06 \\
  56 & 2.92e-11 & 2.22e-08 & 1.20e-10 & 2.93e-08 & 7.57e-09 & 2.76e-06 \\
\hline
rate & 4.11        & 2.93        & 3.92        & 2.99        & 2.98         & 1.99  \\
\hline
\end{tabular}
\caption{Error of $P_2$ IFE Solution with $\beta=[1,5]$, $\alpha = \pi/6$, $\gamma=1$, $c=1$.}
\label{table: general elliptic P2-IFE-1-5}
\end{small}
\end{center}
\end{table}

\begin{table}[thb]
\begin{center}
\begin{small}
\begin{tabular}{|r|c|c|c|c|c|c|}
\hline
$1/h$ & $\|e_h\|_{node}$ & $\|e_h\|_{L^\infty}$ &$\|e_h\|_{Lob}$ & $\|\beta e_h'\|_{Gau}$ &$\|e_h\|_{L^2}$ & $\|e_h\|_{H^1}$\\
\hline
    8 & 4.24e-10 & 1.18e-07 & 1.65e-09 & 6.97e-08 & 5.59e-08 & 4.27e-06 \\
  10 & 1.65e-10 & 4.83e-08 & 5.38e-10 & 2.88e-08 & 2.29e-08 & 2.19e-06 \\
  12 & 6.63e-11 & 2.33e-08 & 2.16e-10 & 1.40e-08 & 1.11e-08 & 1.27e-06 \\
  14 & 2.57e-11 & 1.26e-08 & 1.00e-10 & 7.58e-09 & 5.97e-09 & 7.95e-07 \\
  16 & 8.77e-12 & 7.37e-09 & 5.13e-11 & 4.45e-09 & 3.50e-09 & 5.32e-07 \\
  18 & 1.74e-12 & 4.60e-09 & 2.85e-11 & 2.78e-09 & 2.18e-09 & 3.73e-07 \\
  20 & 1.00e-12 & 3.02e-09 & 1.69e-11 & 1.82e-09 & 1.43e-09 & 2.72e-07 \\
\hline
rate & 6.79     & 4.00     & 4.00     & 5.00     & 4.00     & 3.00  \\
\hline\end{tabular}
\caption{Error of $P_3$ IFE Solution with $\beta=[1,5]$, $\alpha = \pi/6$, $\gamma=1$, $c=1$.}
\label{table: general elliptic P3-IFE-1-5}
\end{small}
\end{center}
\end{table}

Next we illustrate superconvergence behavior at roots of (generalized) orthogonal polynomials. In Figures \ref{fig: error P1 IFE},  \ref{fig: error P2 IFE}, and \ref{fig: error P3 IFE}, we list the plots of solution error $u_h-u$ and the flux error $\beta u_h'-\beta u'$ on the mesh consisting of eight elements. Also, we highlight the roots of corresponding orthogonal polynomials by star with red color. Clearly, at those points, errors are much smaller compared to other points. Note that the interface $\alpha\in (0.5,0.6)$, and the red-color-marked points on this interface element are roots of generalized Lobatto/Legendre polynomials. 

\begin{figure}[htb]
  \centering
  \includegraphics[width=.48\textwidth]{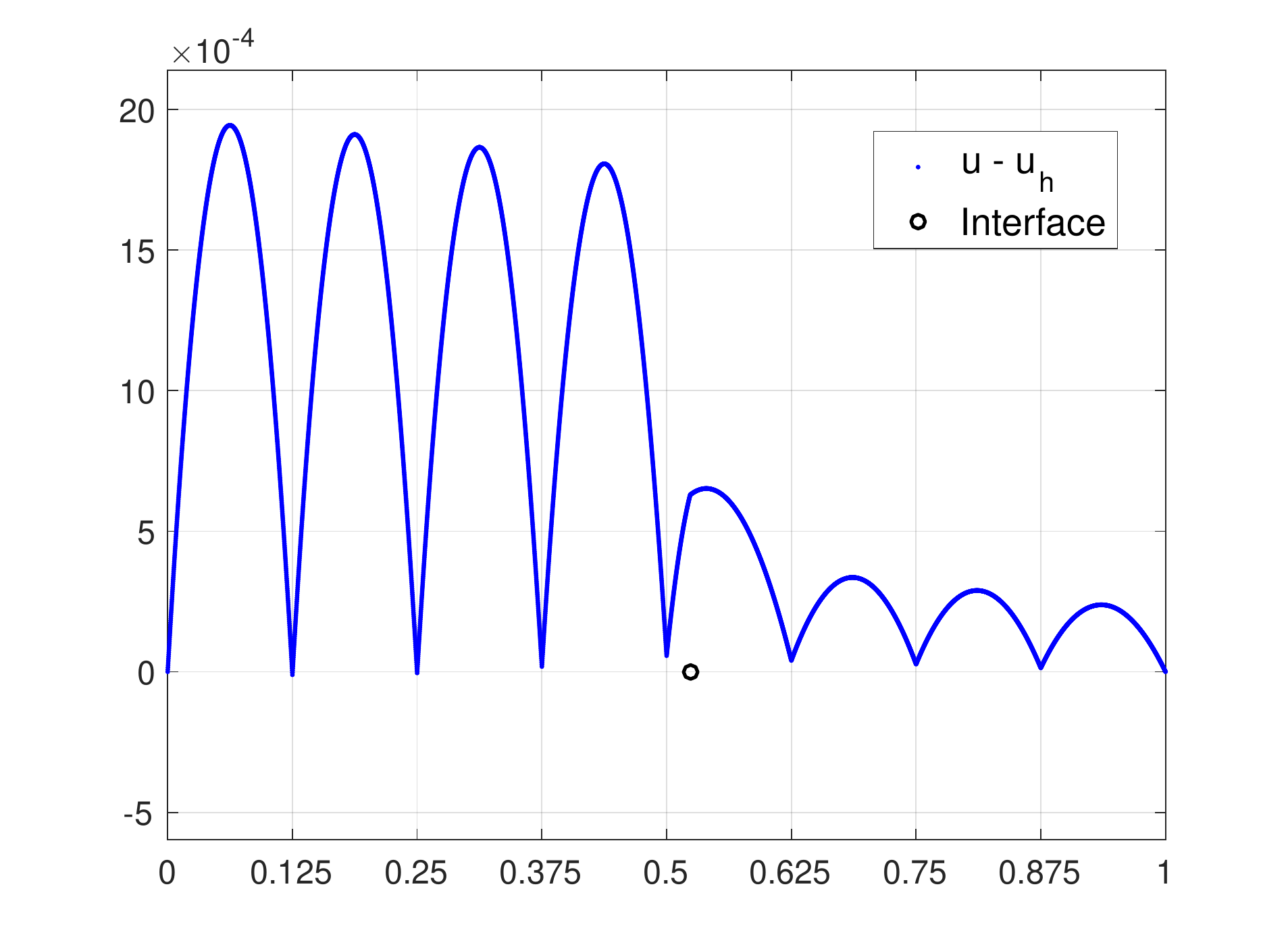}~
  \includegraphics[width=.48\textwidth]{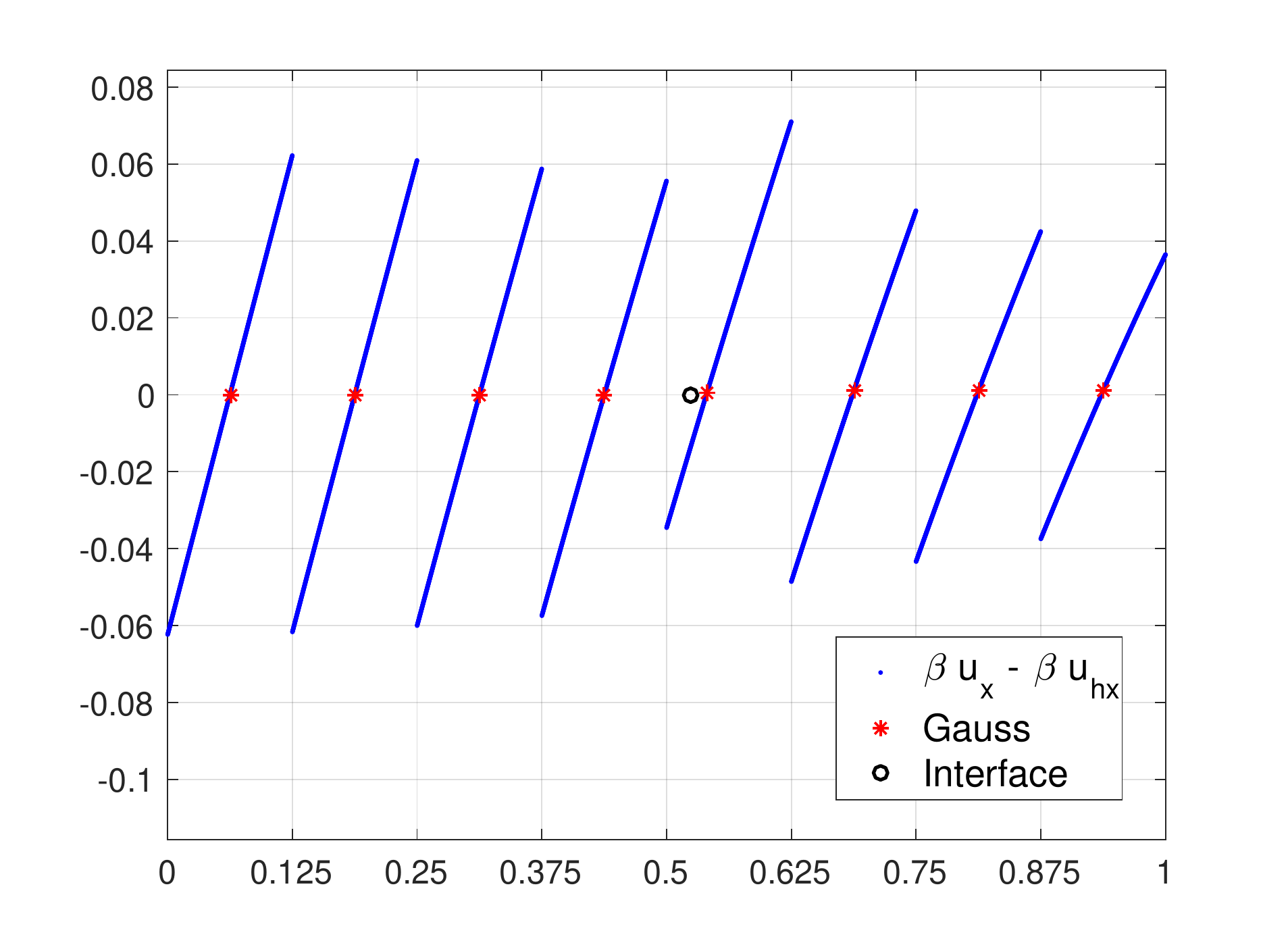}
  \caption{Error and flux error of $P_1$ IFE solution. $\beta = [1,5]$, $\alpha = \pi/6$, $\gamma=1$, $c=1$.}
  \label{fig: error P1 IFE}
\end{figure}

\begin{figure}[htb]
  \centering
  \includegraphics[width=.48\textwidth]{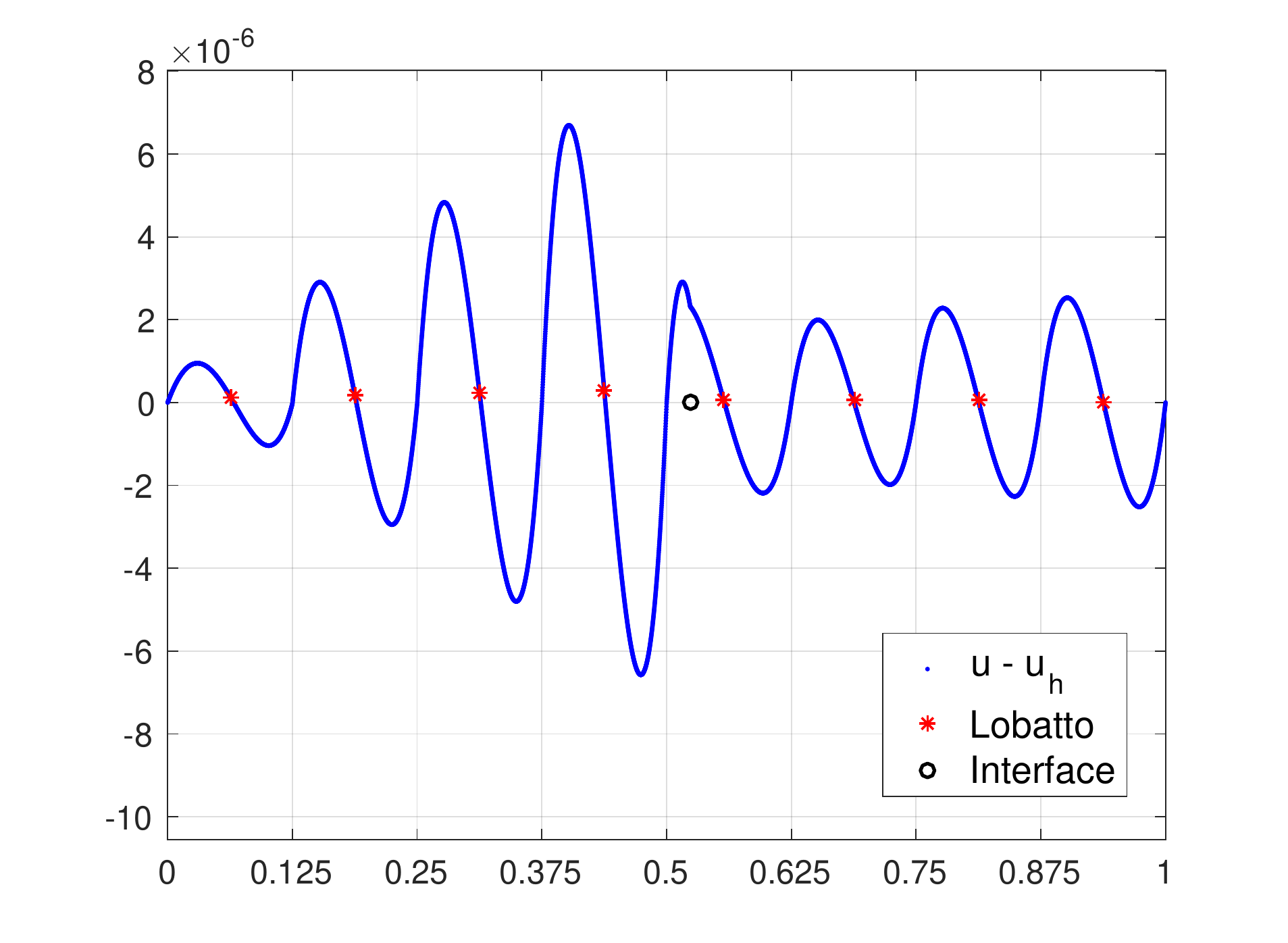}~
  \includegraphics[width=.48\textwidth]{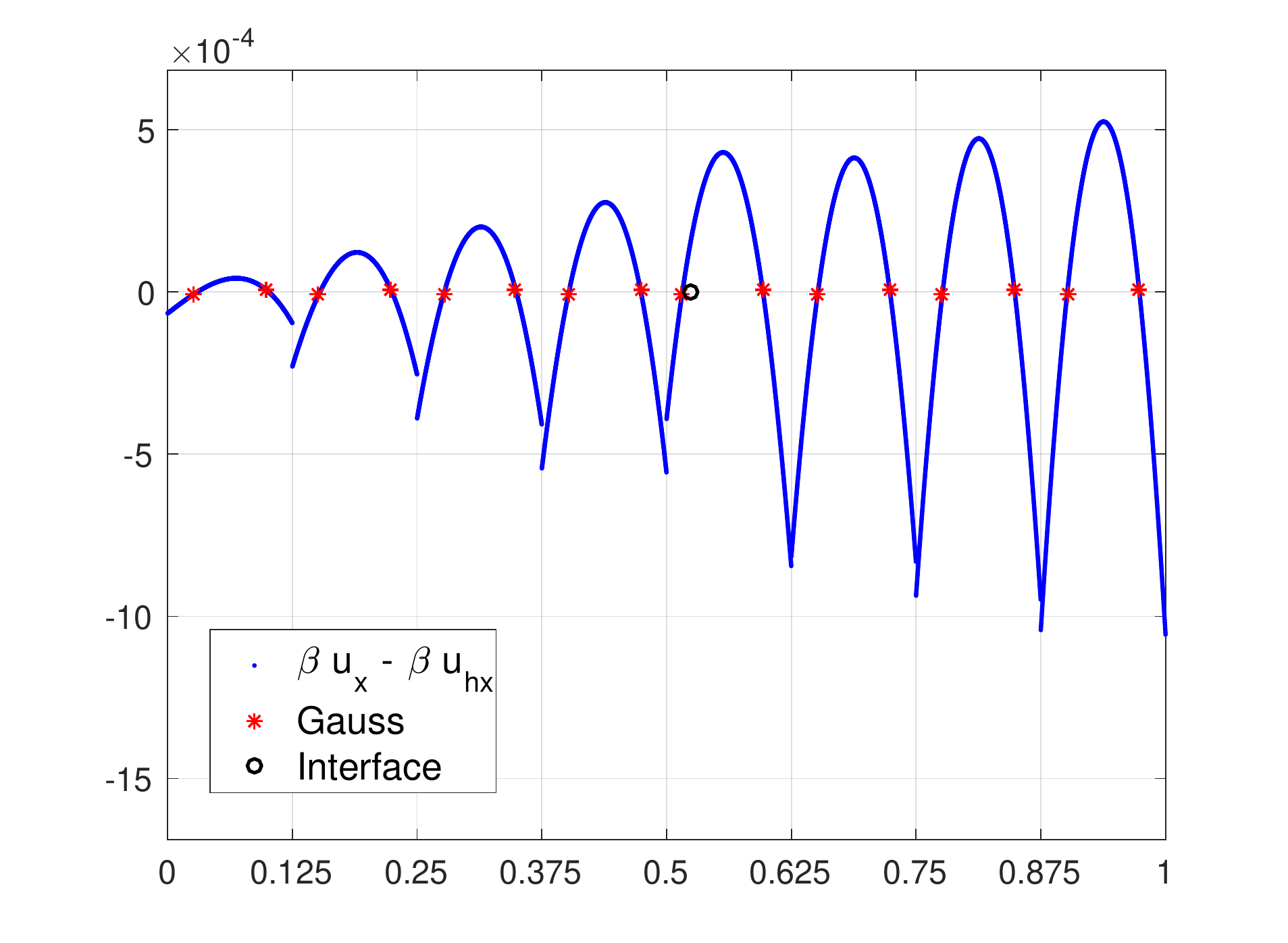}
  \caption{Error and flux error of $P_2$ IFE solution. $\beta = [1,5]$, $\alpha = \pi/6$, $\gamma=1$, $c=1$.}
  \label{fig: error P2 IFE}
\end{figure}

\begin{figure}[htb]
  \centering
  \includegraphics[width=.48\textwidth]{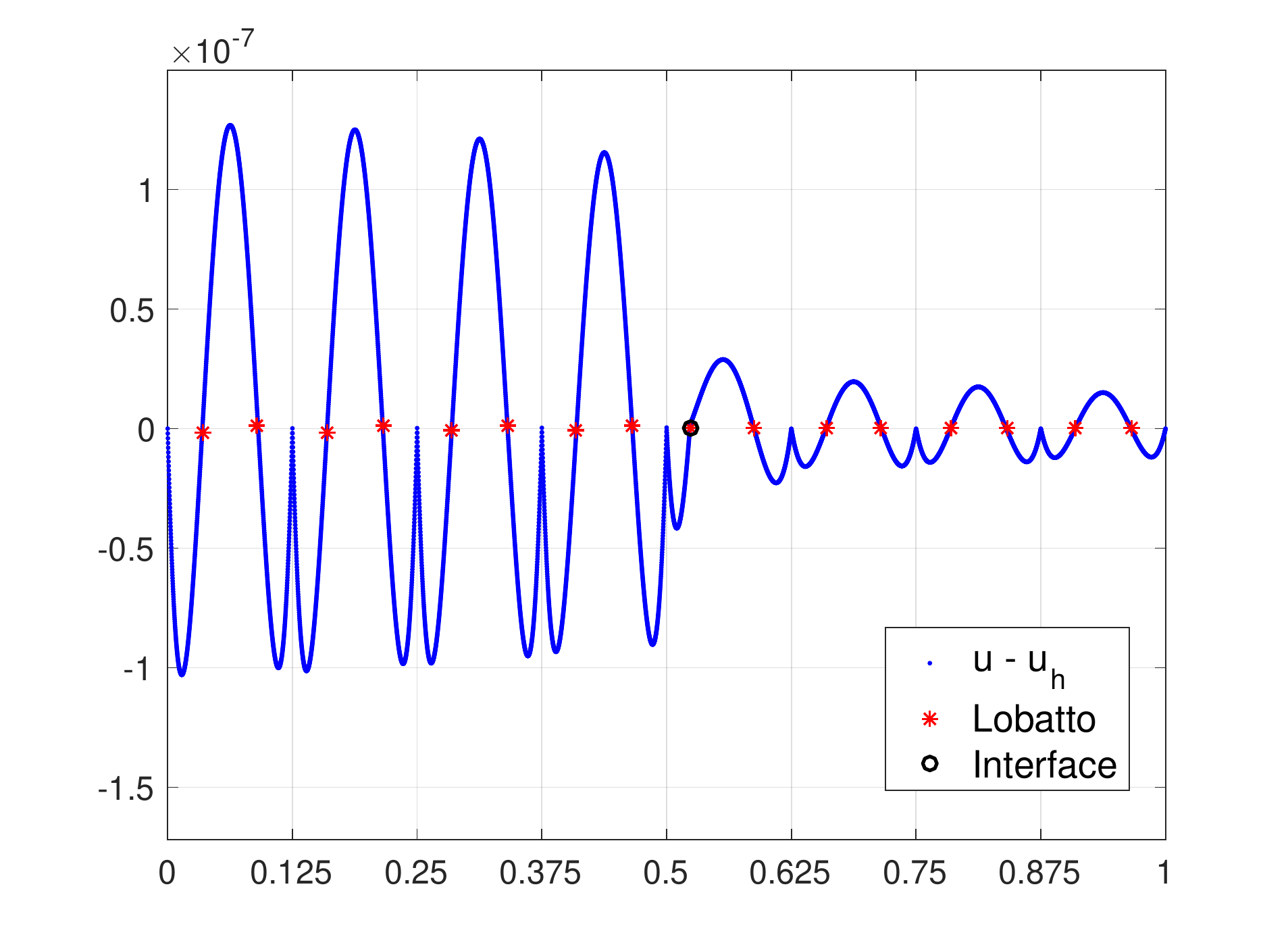}~
  \includegraphics[width=.48\textwidth]{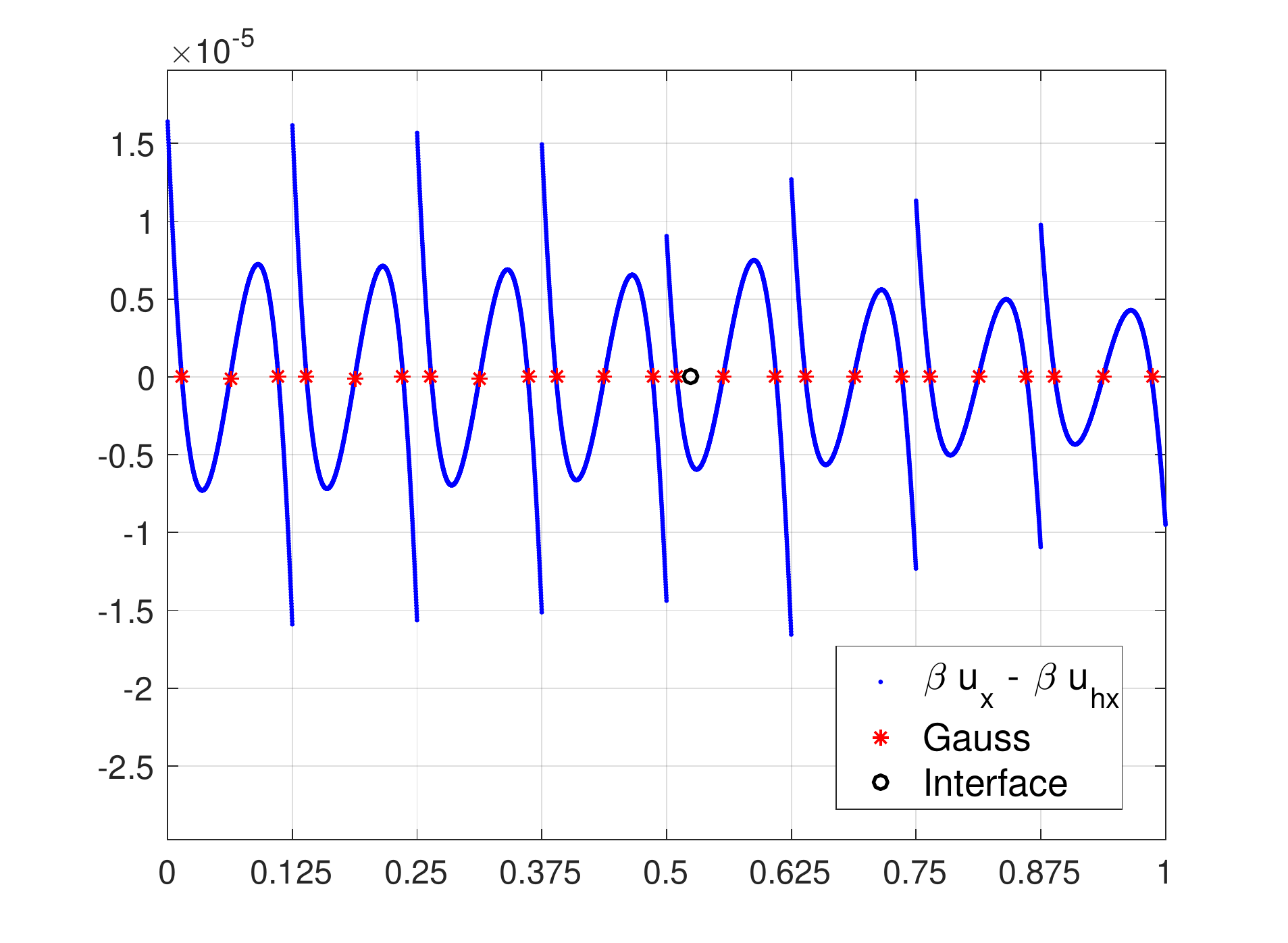}
  \caption{Error and flux error of $P_3$ IFE solution. $\beta^- = 1$, $\beta^+ = 10$, $\alpha = \pi/6$, $\gamma=1$, $c=10$.}
  \label{fig: error P3 IFE}
\end{figure}

\begin{example}\textbf(Multiple interface points)
In this example, we use IFE method to interface problems with multiple discontinuities. In particular, we consider the following function as the exact solution, where the coefficient function $\beta$ has two discontinuities at $\alpha_1$ and $\alpha_2$. 
\begin{equation}\label{eq: IFE ex2}
  u(x) =
  \left\{
    \begin{array}{ll}
      \dfrac{1}{\beta_1}\cos(x), & \text{if}~~x \in [0,\alpha_1), \\
      \dfrac{1}{\beta_2}\cos(x) + \left(\dfrac{1}{\beta_1} - \dfrac{1}{\beta_2}\right)\cos(\alpha_1), & \text{if}~~x \in (\alpha_1,\alpha_2], \\
      \dfrac{1}{\beta_3}\cos(x) + \left(\dfrac{1}{\beta_1} - \dfrac{1}{\beta_2}\right)\cos(\alpha_1) + \left(\dfrac{1}{\beta_2} - \dfrac{1}{\beta_3}\right)\cos(\alpha_2) , & \text{if}~~x \in (\alpha_2,1]. \\
    \end{array}
  \right.
\end{equation}
We set the interface points $\alpha_1 = \frac{\pi}{6}$, and $\alpha_2 =  \frac{\pi}{6}+0.06$. They separate the domain into three subdomiains, on which the diffusion coefficients are chosen as $\beta_1 = 1$, $\beta_2 = 5$, $\beta_3 = 100$.
It can be easy to verify that
\begin{equation*}
  \jump{u(\alpha_i)} = 0,~~~
  \bigjump{\beta u^{(j)}(\alpha_i)} = 0, ~~\forall j \geq 1, ~i = 1,2.
\end{equation*}
\end{example}

\begin{table}[thb]
\begin{center}
\begin{small}
\begin{tabular}{|r|c|c|c|c|c|}
\hline
$1/h$ & $\|e_h\|_{node}$ & $\|e_h\|_{L^\infty}$ & $\|\beta e_h'\|_{Gau}$ &$\|e_h\|_{L^2}$ & $\|e_h\|_{H^1}$\\
\hline
    8 & 2.71e-05 & 1.92e-03 & 1.38e-03 & 9.67e-04 & 2.46e-02 \\
  16 & 5.26e-06 & 4.81e-04 & 3.48e-04 & 2.42e-04 & 1.24e-02 \\
  32 & 1.46e-06 & 1.20e-04 & 8.78e-05 & 6.06e-05 & 6.20e-03 \\
  64 & 3.86e-07 & 3.01e-05 & 2.20e-05 & 1.52e-05 & 3.11e-03 \\
128 & 1.02e-07 & 7.53e-06 & 5.52e-06 & 3.82e-06 & 1.56e-03 \\
256 & 2.56e-08 & 1.88e-06 & 1.38e-06 & 9.55e-07 & 7.81e-04 \\
512 & 6.40e-09 & 4.71e-07 & 3.45e-07 & 2.39e-07 & 3.91e-04 \\
\hline
rate & 1.98     & 2.00     & 1.99     & 2.00    & 1.00  \\
\hline\end{tabular}
\caption{Error of $P_1$ IFE Solution with $\beta=\{1,5,100\}$, $\alpha = \{\dfrac{\pi}{6},\dfrac{\pi}{6}+0.06\}$, $\gamma=1$, $c=1$.}
\label{table: general elliptic 2pt P1-IFE}
\end{small}
\end{center}
\end{table}

\begin{table}[thb]
\begin{center}
\begin{small}
\begin{tabular}{|r|c|c|c|c|c|c|}
\hline
$1/h$ & $\|e_h\|_{node}$ & $\|e_h\|_{L^\infty}$ &$\|e_h\|_{Lob}$ & $\|\beta e_h'\|_{Gau}$ &$\|e_h\|_{L^2}$ & $\|e_h\|_{H^1}$\\
\hline
    8 & 2.89e-08 & 6.70e-06 & 3.08e-07 & 9.77e-06 & 2.23e-06 & 1.17e-04 \\
  16 & 6.26e-09 & 8.84e-07 & 1.54e-08 & 1.45e-06 & 2.89e-07 & 3.05e-05 \\
  24 & 1.36e-09 & 2.67e-07 & 2.95e-09 & 3.67e-07 & 8.66e-08 & 1.35e-05 \\
  32 & 2.06e-10 & 1.14e-07 & 1.17e-09 & 1.55e-07 & 3.62e-08 & 7.53e-06 \\
  40 & 4.12e-11 & 5.87e-08 & 5.46e-10 & 9.17e-08 & 1.90e-08 & 4.93e-06 \\
  48 & 2.69e-11 & 3.54e-08 & 2.60e-10 & 4.67e-08 & 1.11e-08 & 3.46e-06 \\
  56 & 3.50e-11 & 2.22e-08 & 1.15e-10 & 2.93e-08 & 6.94e-09 & 2.53e-06 \\
\hline
rate & 3.95     & 2.93     & 3.95     & 2.97     & 3.00     & 1.97  \\
\hline
\end{tabular}
\caption{Error of $P_2$ IFE Solution with $\beta=\{1,5,100\}$, $\alpha = \{\dfrac{\pi}{6},\dfrac{\pi}{6}+0.06\}$, $\gamma=1$, $c=1$.}
\label{table: general elliptic 2pt P2-IFE}
\end{small}
\end{center}
\end{table}

\begin{table}[thb]
\begin{center}
\begin{small}
\begin{tabular}{|r|c|c|c|c|c|c|}
\hline
$1/h$ & $\|e_h\|_{node}$ & $\|e_h\|_{L^\infty}$ &$\|e_h\|_{Lob}$ & $\|\beta e_h'\|_{Gau}$ &$\|e_h\|_{L^2}$ & $\|e_h\|_{H^1}$\\
\hline
    8 & 2.01e-10 & 1.18e-07 & 1.66e-09 & 6.94e-08 & 5.51e-08 & 4.20e-06 \\
  10 & 9.06e-10 & 4.83e-08 & 5.41e-10 & 2.87e-08 & 2.26e-08 & 2.15e-06 \\
  12 & 7.44e-11 & 2.33e-08 & 2.16e-10 & 1.40e-08 & 1.10e-08 & 1.25e-06 \\
  14 & 2.94e-11 & 1.26e-08 & 9.99e-11 & 7.58e-09 & 5.92e-09 & 7.88e-07 \\
  16 & 1.00e-11 & 7.37e-09 & 5.13e-11 & 4.45e-09 & 3.47e-09 & 5.27e-07 \\
  18 & 1.70e-12 & 4.60e-09 & 2.85e-11 & 2.78e-09 & 2.16e-09 & 3.70e-07 \\
  20 & 1.27e-12 & 3.02e-09 & 1.69e-11 & 1.82e-09 & 1.42e-09 & 2.70e-07 \\
\hline
rate & 5.76     & 4.00     & 5.01     & 3.97     &  3.99    & 3.00  \\
\hline\end{tabular}
\caption{Error of $P_3$ IFE Solution with $\beta=\{1,5,100\}$, $\alpha = \{\dfrac{\pi}{6},\dfrac{\pi}{6}+0.06\}$, $\gamma=1$, $c=1$.}
\label{table: general elliptic 2pt P3}
\end{small}
\end{center}
\end{table}

\begin{figure}[htb]
  \centering
  \includegraphics[width=.48\textwidth]{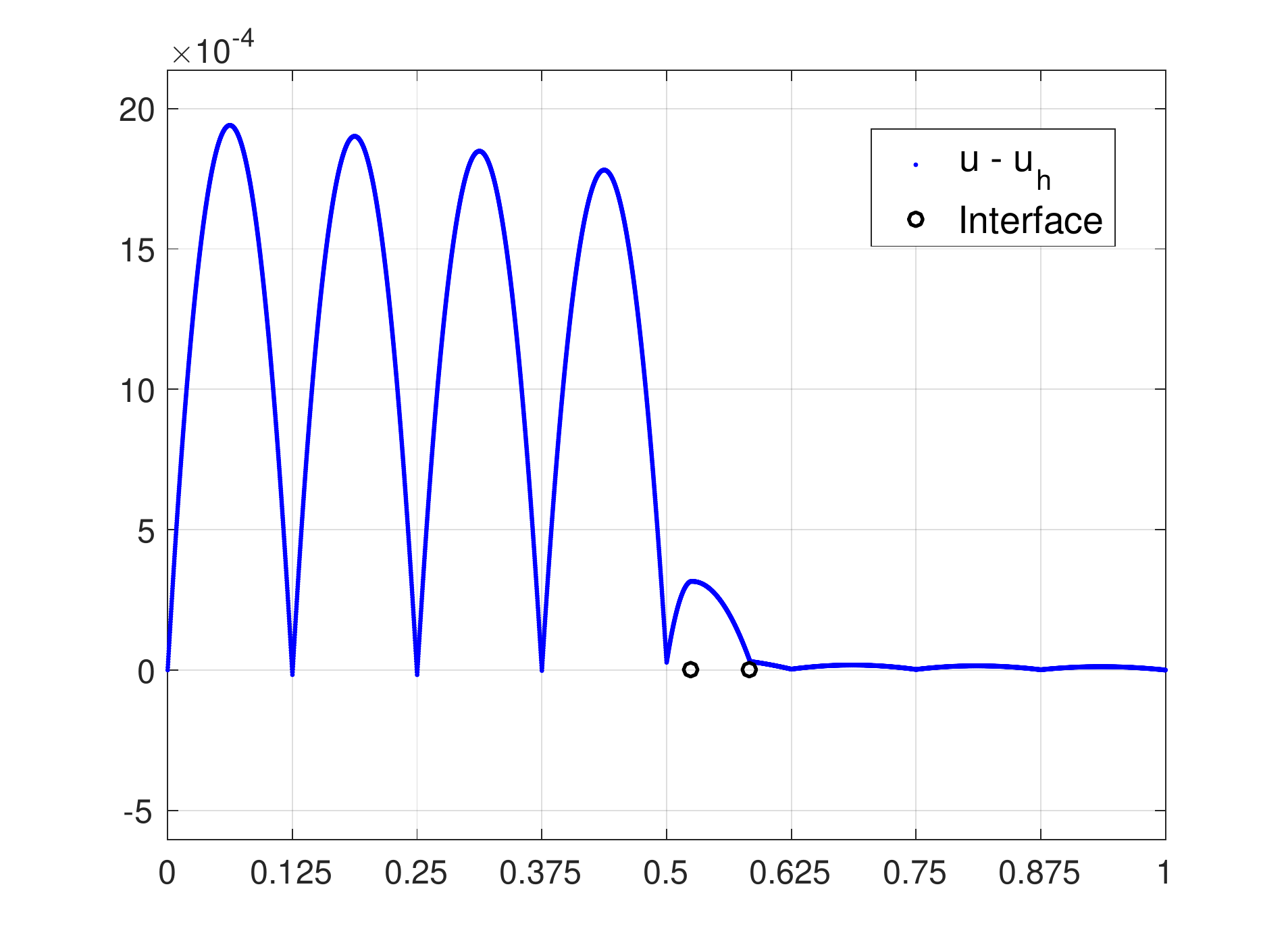}~
  \includegraphics[width=.48\textwidth]{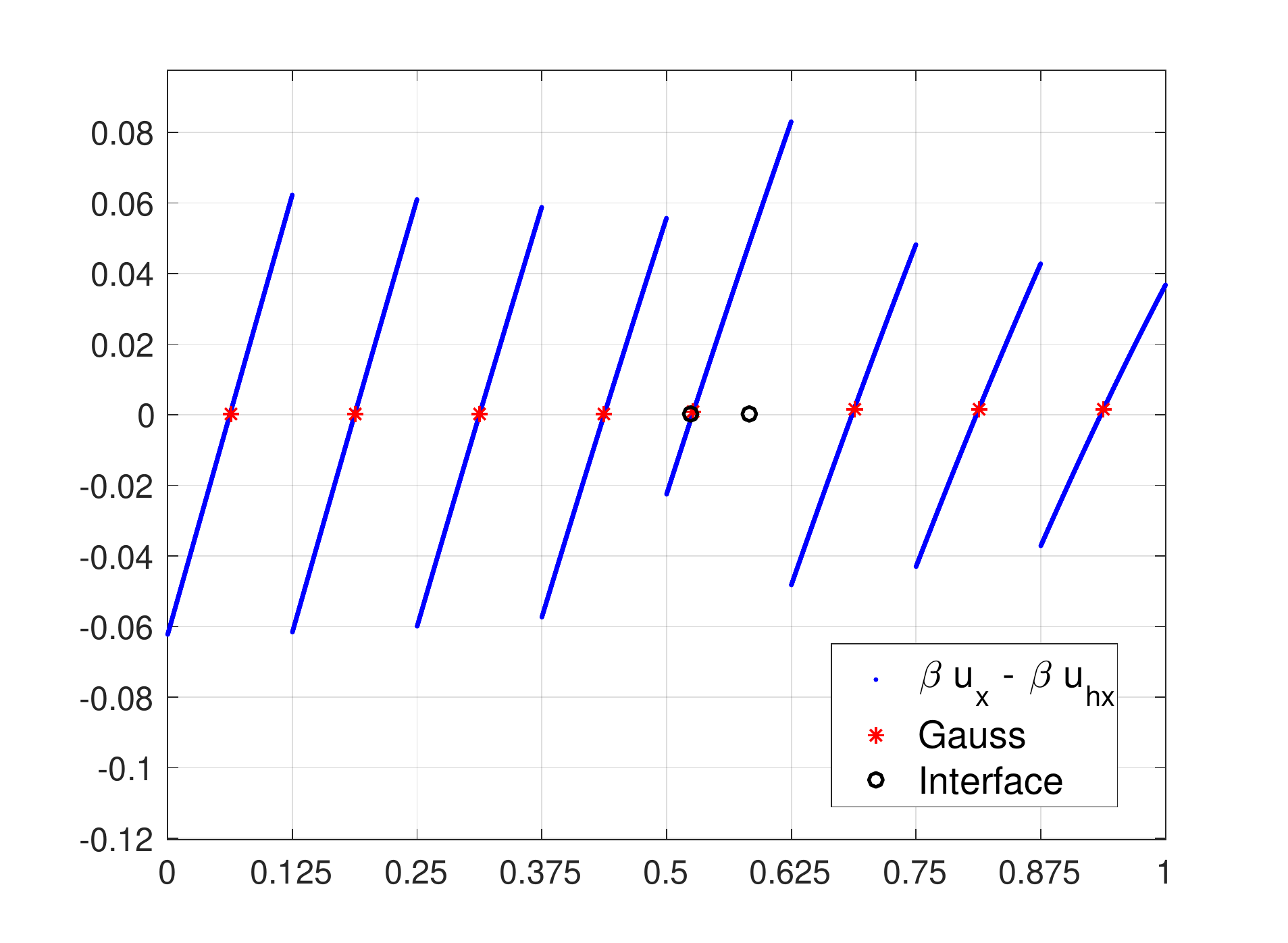}
  \caption{Error and flux error of $P_1$ IFE solution. $\beta=\{1,5,100\}$, $\alpha = \{\dfrac{\pi}{6},\dfrac{\pi}{6}+0.06\}$}
  \label{fig: error P1 IFE 2pt}
\end{figure}

\begin{figure}[htb]
  \centering
  \includegraphics[width=.48\textwidth]{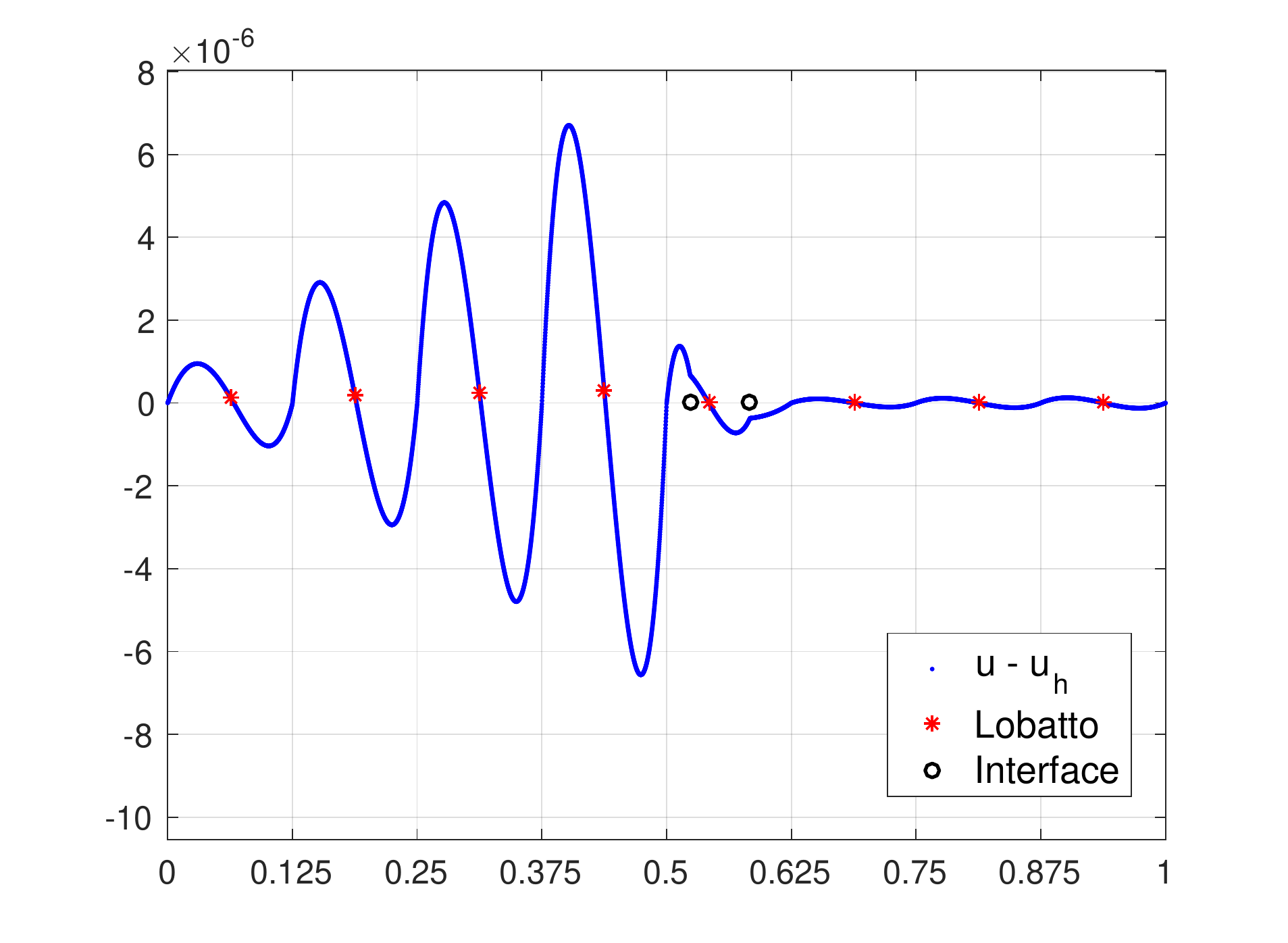}~
  \includegraphics[width=.48\textwidth]{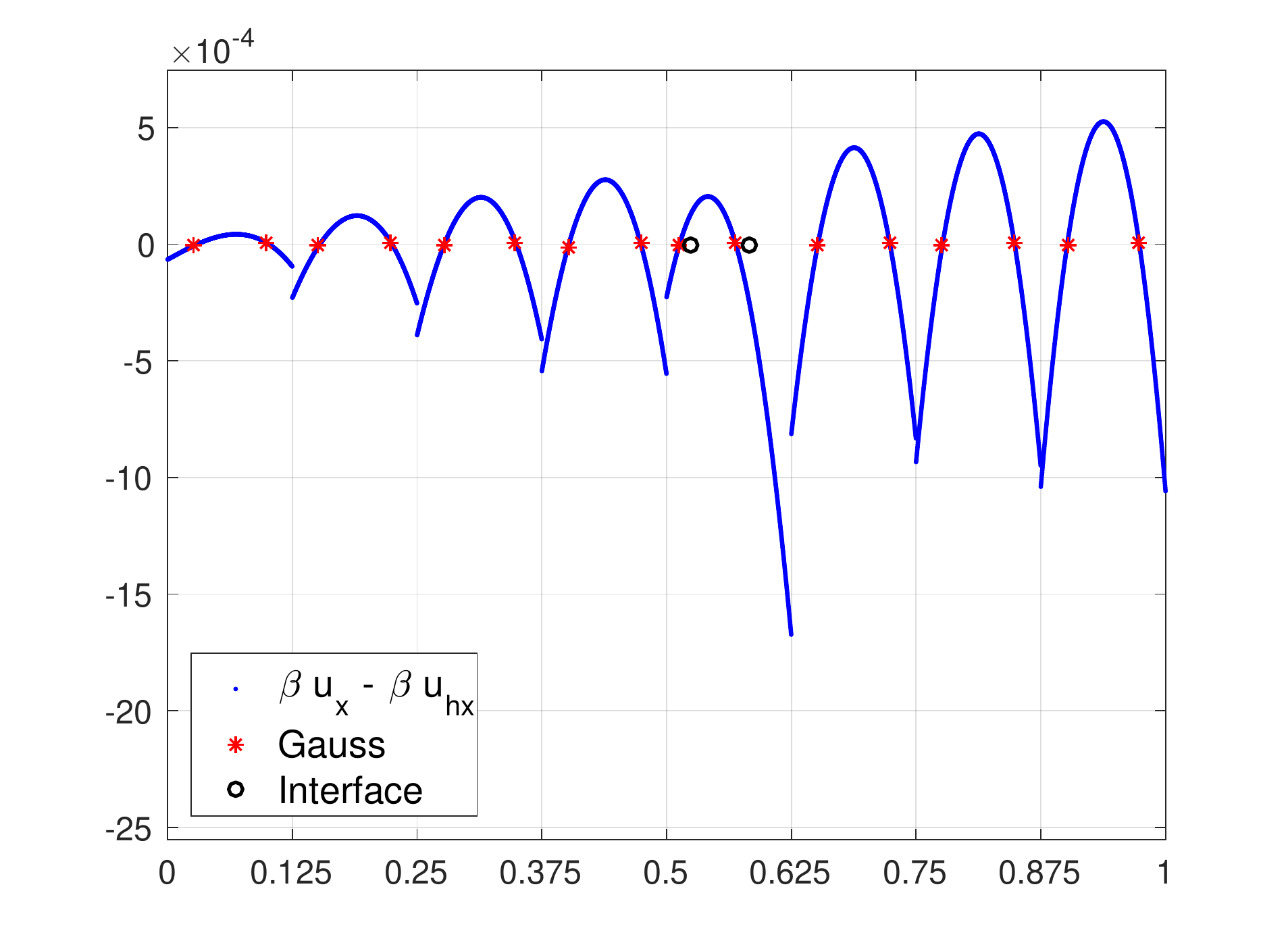}
  \caption{Error and flux error of $P_2$ IFE solution. $\beta=\{1,5,100\}$, $\alpha = \{\dfrac{\pi}{6},\dfrac{\pi}{6}+0.06\}$}
  \label{fig: error P2 IFE 2pt}
\end{figure}

\begin{figure}[htb]
  \centering
  \includegraphics[width=.48\textwidth]{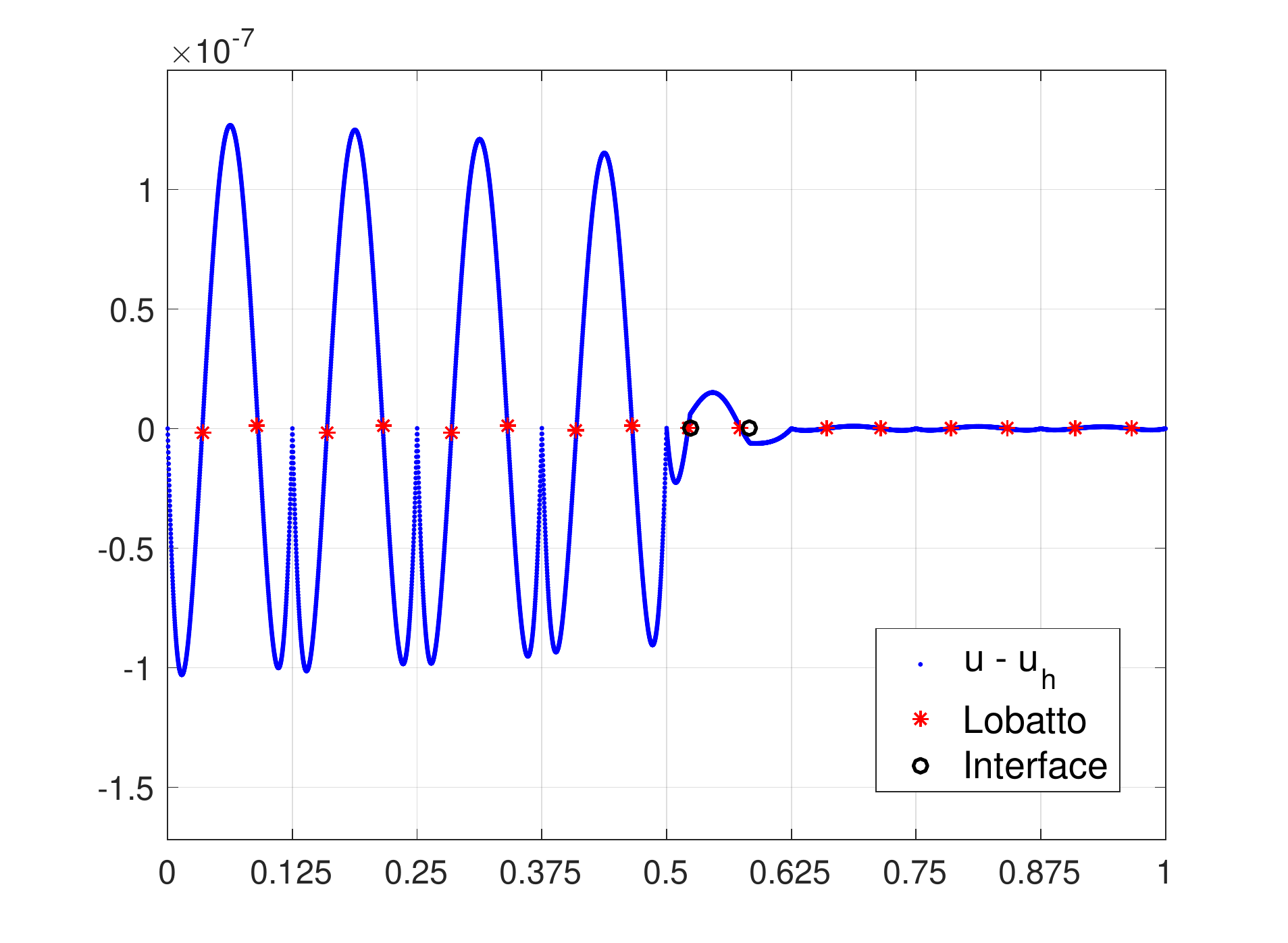}~
  \includegraphics[width=.48\textwidth]{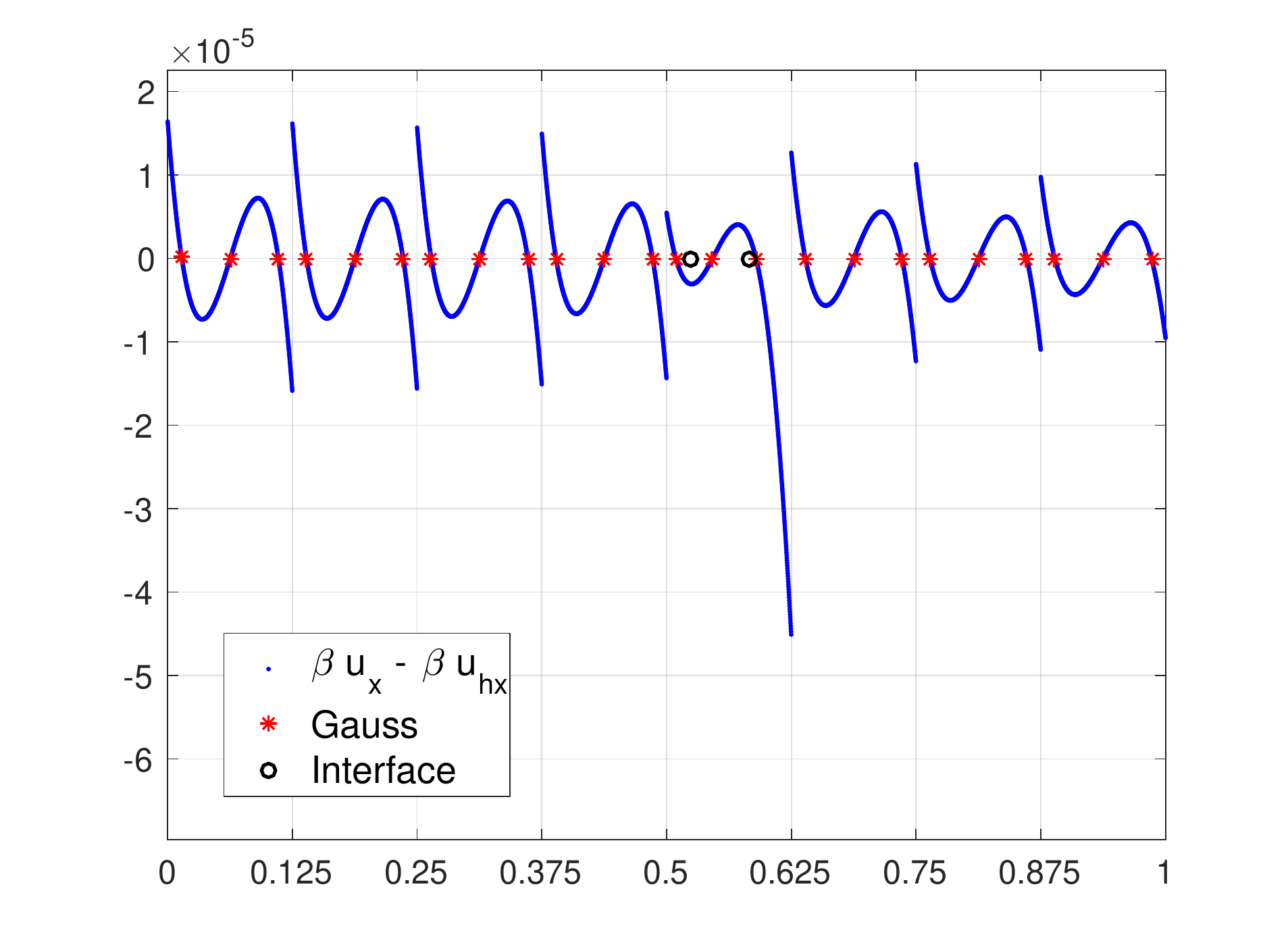}
  \caption{Error and flux error of $P_3$ IFE solution. $\beta=\{1,5,100\}$, $\alpha = \{\dfrac{\pi}{6},\dfrac{\pi}{6}+0.06\}$}
  \label{fig: error P3 IFE 2pt}
\end{figure}

Tables \ref{table: general elliptic 2pt P1-IFE} - \ref{table: general elliptic 2pt P3} report the numerical errors and convergence rates in different norms. Figures \ref{fig: error P1 IFE 2pt} - \ref{fig: error P3 IFE 2pt} demonstrate the superconvergence behavior on the roots of generalized Lobatto/Legendre polynomials. We note that, on the coarsest mesh which contains $8$ elements, the interface element contains two interface points. As the mesh size becomes smaller, the interface points are separated in different elements. This example shows the robustness of our scheme with respect to multiple coefficient discontinuities.

The numerical results for diffusion (only) interface problems are similar, except at mesh points there are only roundoff errors. We also conducted numerical experiments for different configuration of  interface locations $\alpha$, and different sets of coefficients $\beta^\pm$, including large coefficient contrast. Similar superconvergence properties have been observed as the exemplified examples, hence we omit these data in the article.

\section{Conclusion}
In this article, we developed explicitly, the orthogonal IFE basis functions. First we constructed a set of bases for flux using (generalized) Legendre polynomials, then integrate to obtain basis functions for the primary unknown. The procedure is somewhat ``reversed" from the classical approach in constructing IFE basis functions.
The superconvergence behavior has been observed and proved for general elliptic interface problems in the one dimensional setting. At the roots of generalized Lobatto polynomial of degree $p+1$, the IFE solution is superconvergent to the exact solution with order $p+2$ (comparing with the optimal order $p+1$); at the roots of generalized Legendre polynomial of degree $p$, the derivative of the IFE solution is superconvergent to the derivative of the exact solution with order $p+1$ (comparing with the optimal order $p$). In addition, the convergent rate at all mesh points (including those of the interface element) is of order $p+2$ (comparing with the optimal order $p+1$). The idea presented in this article seems extendable to the two dimensional elliptic interface problems (at least for the tensor-product space case), which will be of interesting in future work.

%\bibliographystyle{abbrv}
%%
%\bibliography{xuzhangBib.bib}
\section*{Acknowledgement}
The authors would like to thank Prof. Tao Lin for his valuable suggestions on this article.

\end{document}